\documentclass[a4paper,twoside,reqno]{amsart}
\usepackage[cp1251]{inputenc}
\usepackage{amsmath,amssymb,graphicx,color}
\usepackage[colorlinks,pdfauthor={Yu. Brezhnev},bookmarksopen=false,
            bookmarks=false,pdfwindowui=false]{hyperref}

\usepackage{hyperref}

\newcommand{\textcenter}[3][Yu.~Brezhnev]{
\pagestyle{myheadings}
\markboth{\hfill\textsc{#1}\hfill}{\hfill\textsc{#1}\hfill}
\ifdim#2mm=0mm\textwidth=\textwidth\else\textwidth=#2mm\fi
\oddsidemargin=\paperwidth%
\addtolength{\oddsidemargin}{-\textwidth}%
\addtolength{\oddsidemargin}{-2in} \oddsidemargin=0.5\oddsidemargin
\evensidemargin=\oddsidemargin
\ifdim#3mm=0mm\textheight=\textheight\else\textheight=#3mm\fi
\topmargin=\paperheight%
\addtolength{\topmargin}{-\textheight}%
\addtolength{\topmargin}{-2in}%
\addtolength{\topmargin}{-\headsep}%
\addtolength{\topmargin}{-\headheight}%
\addtolength{\topmargin}{-\footskip}%
\topmargin=0.5\topmargin}

 \textcenter{139.1}{200} 

\theoremstyle{plain}
\newtheorem{theorem}{Theorem}
\newtheorem{lemma}[theorem]{Lemma}
\newtheorem{proposition}[theorem]{Proposition}

\theoremstyle{remark}
\newtheorem{remark}{Remark}

\allowdisplaybreaks[4] 

\def~{\unskip\nobreak\hspace{0.27778em}\ignorespaces}
\def\,{\ifmmode\mskip+1.5mu\else\kern+.08333em\fi\relax}
\def\!{\ifmmode\mskip-1.5mu\else\kern-.08333em\fi\relax}
\newdimen{\FontSize} \delimiterfactor=990

\def\ds{\displaystyle}  \def\sss{\scriptscriptstyle}
\def\ie{i.\,\,e.}  \def\eg{e.\,\,g.}
\def\re{{\,\mathrm{e}}} \def\ri{{\mathrm{i}}}
\def\wpp{\wp\hbox{\smaller[1]$'$}}

\def\DEF{\mathrel{\vcenter{\hbox{$:$}}{=}}}
\def\FED{\mathrel{{=}\vcenter{\hbox{$:$}}}}

\def\smin{\mathrel{\vcenter{\hbox{\scalebox{0.7}[1]{$\scriptstyle-$}}}}}
\def\hence{=\kern-0.5em\Rightarrow}
\def\?{\textrm{\protect\footnotesize$\RED\mathchar"446$}}
\def\vphi{\mskip0mu\raisebox{-0.54ex}{\scalebox{0.6}[1]{\hbox{$-$}}}
 \mkern-7.76mu\raise0.52ex\hbox{$\varphi$}}
\def\vpsi{\mskip0.4mu\raisebox{-0.54ex}{\scalebox{0.6}[1]{\hbox{$-$}}}
 \mkern-8.82mu\raise0.52ex\hbox{$\psi$}{}}

\def\END{\end{document}}

\def\Mfrac#1#2{\def\FRAC{\hbox{\smaller[1]$\ds\frac{#1}{#2}$}}
\ifdim\FontSize=11pt\raise0.051ex\FRAC\else\raise0.059ex\FRAC\fi}
\def\mfrac#1#2{\def\FRAC{\hbox{\smaller[3]$\ds\frac{#1}{#2}$}}
\ifdim\FontSize=11pt\raise0.13ex\FRAC\else\raise0.174ex\FRAC\fi}
\newcommand{\AB}[2]{\nobreak\mskip-1mu
\raise0.12ex\hbox{\smaller[2]\renewcommand{\arraystretch}{0.42}%
\setlength{\arraycolsep}{0pt}\raise0.16ex\hbox{$\big[$}%
$\begin{array}{c}\scriptstyle#1\\ \scriptstyle#2\end{array}$%
\raise0.16ex\hbox{$\big]$}}\mskip-1mu\relax}

\newcommand{\mbig}[2][4]{\vcenter{\hbox{%
\ifcase#1$#2$\or\small$\big#2$\or$\big#2$\or\large$\big#2$%
\or\footnotesize$\Big#2$\or\small$\Big#2$\or$\Big#2$\or\large$\Big#2$%
\or$\bigg#2$\or\large$\bigg#2$\or$\Bigg#2$\or\large$\Bigg#2$%
\or\Large$\Bigg#2$\else\huge$\bigg#2$\fi}}\relax}

\newdimen{\upp}\newdimen{\D}\newdimen{\DD} 
\newcommand{\pow}[4][\displaystyle]{
\settowidth{\upp}{\hbox{$#1 {}^{#4}$}}%
\settowidth{\D}{\hbox{$#1{#2}$}}%
\settowidth{\DD}{\hbox{$#1{#2}_{#3}$}}%
\addtolength{\DD}{-\D} {#1{#2}_{#3}\kern-\DD {}^{#4}}
\addtolength{\DD}{-\upp}\ifdim\DD>0pt \kern+\DD \fi\relax}

\makeatletter 
\providecommand{\larger}[1][1]{
\ifmmode\relax\else\@tempcnta\ifx\@currsize\normalsize 4\else
 \ifx\@currsize\small 3\else\ifx\@currsize\footnotesize 2\else
 \ifx\@currsize\large 5\else\ifx\@currsize\Large 6\else
 \ifx\@currsize\LARGE 7\else\ifx\@currsize\scriptsize 1\else
 \ifx\@currsize\tiny  0\else\ifx\@currsize\huge 8\else
 \ifx\@currsize\Huge  9\else 4 \fi\fi\fi\fi\fi\fi\fi\fi\fi\fi
\advance\@tempcnta#1\relax\ifnum\@tempcnta<\z@ \@tempcnta\z@ \fi
\ifcase\@tempcnta\tiny\or\scriptsize\or\footnotesize\or\small
 \or\normalsize\or\large\or\Large\or\LARGE\or\huge\else\Huge\fi\fi}

\providecommand{\smaller}[1][1]{\larger[-#1]} \makeatother

\makeatletter \def\@dotsep{1} \makeatother 

\def\RED{\color{red}} \def\BLUE{\color{blue}} 

\newcommand{\Width}[2][\displaystyle]
{\settowidth\D{\hbox{$#1#2$}}\texttt{\tiny\BLUE$($\the\D$)$}}
\newcommand{\Height}[2][\displaystyle]
{\settoheight\D{\hbox{$#1#2$}}\texttt{\tiny\BLUE$($\the\D$)$}}
\newcommand{\Depth}[2][\displaystyle]
{\settodepth\D{\hbox{$#1#2$}}\texttt{\tiny\BLUE$($\the\D$)$}}

\def\T{\hbox{\small $\mathrm{T}$}}
\def\bI{{\boldsymbol{I}}}
\def\bO{{\boldsymbol{\Omega}}}
\newcommand{\bdot}[1]{\mathbf{\dot{\mathnormal{#1}}}}

\begin{document}
\author[Yu.~Brezhnev]{Yurii V.~Brezhnev}
\author[S.~Lyakhovich]{Simon L.~Lyakhovich}
\author[A.~Sharapov]{Alexey A.~Sharapov}

\title[Dynamical systems defining  $\vartheta$-constants]
{Dynamical systems defining Jacobi's $\vartheta$-constants}



\thanks{The research was supported by the Federal Targeted Program
under state contracts 02.740.11.0238, \#P1337 and \#P22. The work by
SLL and {\sc AASh} was supported by the RFBR grant 09--02--00723-a and
SLL had a partial support from the RFBR grant 08--01--00737-a. SLL and
{\sc AASh} appreciate the hospitality of the Erwin Schroedinger
Institute for Mathematical Physics, Vienna.}


\begin{abstract}
We propose  a system of equations  that defines Weierstrass--Jacobi's
eta- and theta-constant series in a differentially closed way. This
system is shown to have a direct relationship to a little-known
dynamical system obtained by Jacobi. The classically known differential
equations by Darboux--Halphen, Chazy, and Ramanujan are the
differential consequences or reductions of these systems. The proposed
system is shown to admit the Lagrangian, Hamiltonian, and Nambu
formulations. We explicitly construct a pencil of nonlinear Poisson
brackets and complete set of involutive conserved quantities. As
byproducts of the theory, we exemplify conserved quantities for the
Ramamani dynamical system and quadratic system of Halphen--Brioschi.
\end{abstract}

\hfill {\footnotesize Journ.~Math.~Phys. (2011), {\bf52}(11),
112704(1--21)} \medskip

\maketitle\thispagestyle{empty}

\tableofcontents

\section{Introduction}
\noindent In this work we propose a description of the classical
Jacobi's $\vartheta$-constants and Weierstrass' $\eta$-function by
means of closed and Lagrangian/Hamiltonian ordinary differential
equations (ODEs). By simple transformations or reductions these
equations lead to many well-known differential systems. Among these are
the Darboux--Halphen system \cite{darboux,halphen}, some its
modifications \cite{ramanujan,ablowitz3}, the Chazy equation
\cite{conte}, and also a Jacobi system of ODEs \cite{jacobi} which has
not received mention in the modern literature in the context. For both
Jacobi's system and equations defining the $\vartheta,\eta$-series we
work out the Hamiltonian formalism and show that they admit a pencil of
(compatible) Poisson structures in the sense of Magri \cite{magri} and
formulation as the generalized Nambu mechanics \cite{Nambu} with a
certain 4-bracket.

The three Jacobi's theta-constants are defined by the classical series
$$
\vartheta_2(\tau)\DEF\re^{\frac14\pi\ri\tau}_{\mathstrut}
\sideset{}{_k}\sum_{-\infty}^\infty\!
\re^{(k^2+k)\pi\ri\tau}_{\mathstrut},\qquad
\vartheta_3(\tau)\DEF
\sideset{}{_k}\sum_{-\infty}^\infty\!
\re^{k^2\pi\ri\tau}_{\mathstrut} ,\qquad
\vartheta_4(\tau)\DEF
\sideset{}{_k}\sum_{-\infty}^\infty\!
(-1)^k\re^{k^2\pi\ri\tau}_{\mathstrut}
$$
and the Weierstrass $\eta$-function is defined by the series
$$
\eta(\tau)\DEF 2\pi^2\bigg\{ \frac{1}{24}- \sideset{}{_k}\sum_1^\infty
\frac{\re^{2k\pi\ri\,\tau}}
{(1-\re^{2k\pi\ri\,\tau})^2}\bigg\}.
$$
Here, the `time' $\tau$ is considered to be a complex variable
belonging to the upper half-plane $\mathbb{H}^+$:
$\boldsymbol{\Im}(\tau)>0$. These series appear in various problems of
theoretical physics because of their numerous and deep differential
properties \cite{ablowitz3,conte}. Let us mention some of them.

Three $\vartheta$-constant series satisfy the following differential
identities for logarithmic derivatives of their ratios:
$$
\frac{d}{d\tau}\ln\frac{\vartheta_2}{\vartheta_3}=
\frac{\pi}{4}\,\ri\,\vartheta_4^4\,,
\qquad
\frac{d}{d\tau}\ln\frac{\vartheta_3}{\vartheta_4}=
\frac{\pi}{4}\,\ri\,\vartheta_2^4\,,
\qquad
\frac{d}{d\tau}\ln\frac{\vartheta_2}{\vartheta_4}=
\frac{\pi}{4}\,\ri\,\vartheta_3^4\,.
$$
Yet another and very well-known identity is the sum of logarithmic
derivatives:
$$
\frac{\bdot\vartheta_2}{\vartheta_2}+
\frac{\bdot\vartheta_3}{\vartheta_3}+
\frac{\bdot\vartheta_4}{\vartheta_4}=\frac{3\ri}{\pi}\,\eta
$$
(dot stands for the $\tau$-derivative). If we introduce a notation for
these derivatives, say
\begin{equation}\label{xyz}
(X,Y,Z)\DEF 2\!\left(\frac{\bdot\vartheta_2}{\vartheta_2},\:
\frac{\bdot\vartheta_3}{\vartheta_3},\:
\frac{\bdot\vartheta_4}{\vartheta_4}{\sss\!}\right),
\end{equation}
then the quantities $(X,Y,Z)$ satisfy the 3rd order differential system
\begin{equation}\label{darboux}
\bdot{\!\! X}=(Y+Z)\,X-YZ\,,\qquad \bdot{\!\!
Y}=(X+Z)\,Y-XZ\,,\qquad \bdot{\! Z}=(X+Y)\,Z-XY\,,
\end{equation}
which is widely known as the Halphen system
\cite[p.~330--331]{halphen}. This system is frequently named as the
Darboux--Halphen system though Darboux himself wrote down only
differentials \cite[p.~149]{darboux}:
$$
\mathrm{C}\,(d\mathrm{A}+d\mathrm{B})=
\mathrm{B}\,(d\mathrm{A}+d\mathrm{C})=
\mathrm{A}\,(d\mathrm{B}+d\mathrm{C})\,.
$$
These can be written in the form
$$
\frac{d\mathrm{A}}{\mathrm{A}(\mathrm{B}+\mathrm{C})-
\mathrm{B}\mathrm{C}}=
\frac{d\mathrm{B}}{\mathrm{B}(\mathrm{A}+\mathrm{C})-
\mathrm{A}\mathrm{C}}=\frac{d\mathrm{C}}{\mathrm{C}
(\mathrm{A}+\mathrm{B})- \mathrm{A}\mathrm{B}}=dt
$$
which is equivalent to the system \eqref{darboux}.

Remarkable applications of Eqs.~\eqref{darboux} were initiated in the
1990's by M.~Ablowitz et all \cite{ablowitz0, ablowitz} in connection
with reductions of the self-dual Yang--Mills equations. These equations
usually provide the main physical motivation for studying both the
$\eta$, $\vartheta$-series and  allied modular objects. However,
applications go beyond the Yang--Mills theory. In succeeding years the
system appeared in the vacuum cosmological Bianchi--IX model
\cite[p.~143, 147]{conte}, \cite[p.~577]{ablowitz3}, \cite{ablowitz},
theory of 2-monopole moduli spaces \cite{hitchin}, and many other areas
of mathematical physics \cite{kiritsis}. System \eqref{darboux} has
also varieties. One of them is the Weierstrass--Halphen dynamical
system for Weierstrass' invariants $g_2^{}$, $g_3^{}$, and
$\eta$-series:
\begin{equation}\label{g2g3}
\frac{d g_2^{}}{d\tau} = \frac{\ri}{\pi}
\big(8\,g_2^{}\,\eta-12\,g_3^{}\big)\,,\qquad \frac{d
g_3^{}}{d\tau}= \frac{\ri}{\pi}\!\!
\left(12\,g_3^{}\,\eta-\frac23\,g_2^2\right),\qquad
\frac{d\eta}{d\tau}=
\frac{\ri}{\pi}\!\!\left(2\,\eta^2-\frac16\,g_2^{}\right).
\end{equation}
It is known that these invariants
$$
g_2^{}(\tau)\DEF 60\,\sum\limits_{n,\,m}
\frac{1}{(2m\tau+2n)^4}\,,\qquad
g_3^{}(\tau)\DEF 140\,\sum\limits_{n,\,m}
\frac{1}{(2m\tau+2n)^6},\qquad\mbig[1]((n,m)\ne(0,0)\mbig[1])
$$
are related to the $\vartheta$-series by the standard polynomial
formulae
\begin{equation}\label{g23}
\begin{aligned}
g_2^{}(\tau)&={\phantom{4}}\frac{\pi^4}{24}
\big\{\vartheta_2^8(\tau)+\vartheta_3^8(\tau)+\vartheta_4^8(\tau)
\big\}\,,\\[0.3em]
g_3^{}(\tau)&=\,\frac{\pi^6}{432}
\big\{\vartheta_2^4(\tau)+\vartheta_3^4(\tau) \big\}
\big\{\vartheta_3^4(\tau)+\vartheta_4^4(\tau) \big\}
\big\{\vartheta_4^4(\tau)-\vartheta_2^4(\tau) \big\}
\end{aligned}
\end{equation}
and the series themselves satisfy the well-known Jacobi identity
\begin{equation}\label{324}
\vartheta_3^4(\tau)=\vartheta_2^4(\tau)+\vartheta_4^4(\tau)\,.
\end{equation}

In different notation and (number-theoretic) definition for function
series, system \eqref{g2g3} is known as the Ramanujan system of
differential equations \cite{ramanujan,chud} for modular forms
\begin{equation}\label{E2k}
E_2(\tau)=\frac{12}{\pi^2}\,\eta(\tau)\,,\qquad
E_4(\tau)=\frac{12}{\pi^4}\,g_2^{}(\tau)\,,\qquad
E_6(\tau)=\frac{216}{\pi^6}\, g_3^{}(\tau)\,.
\end{equation}
Ramanujan's system is sometimes referred to as the Eisenstein system of
differential equations \cite{chud}, though Eisenstein himself had not
derived it \cite{eisenstein}.  Further discussions of and bibliography
to the systems mentioned above can be found in
\cite{ablowitz,ablowitz3, ablowitz2, ablowitz0,conte, guha} and
references therein.

\subsection{Motivation for the work}
Dynamical variables for all the systems above are rationally expressed
through the $\vartheta$-variables. Therefore, the inverse
transformations will involve the multi-valued functions,  as further
examples show. Recently Ablowitz, Chakravarty, and Hahn
\cite{ablowitz2} called attention to yet another instance which is more
interesting and comes from the equations for modular forms on group
$\boldsymbol{\Gamma}_0(2)$. This is the Ramamani system \cite{rama}
(Sect.~\ref{RR}); it was  also considered  in works \cite{maier,guha}.
In this case, the relation between dynamical variables and the
$\vartheta, \eta$-variables is not obvious because it is given by a
duplication of the $\tau$-argument in forms \eqref{E2k}
\cite{ablowitz2,guha}. If we make use of the duplication rules
\begin{equation}\label{dubl}
\eta(2\,\tau)=\frac12\,\eta(\tau)+\frac{\pi^2}{48}
\big\{\vartheta_3^4(\tau)+\vartheta_4^4(\tau) \big\}\,,\quad
g_2^{}(2\,\tau)=-\frac14\,g_2^{}(\tau)+\frac{5\,\pi^4}{192}\,
\big\{\vartheta_3^4(\tau)+\vartheta_4^4(\tau)\big\}^2
\end{equation}
we arrive again at substitutions of a rational type (see Sect.~\ref{RR}
for further details). Though these rules have not appeared explicitly
in the literature known to us, they can be established by standard
techniques. In a $q$-series notation these identities can be found in
Ramanujan's notebooks and (in number-theoretic notation) implicitly
have been tabulated in \cite[Table~1]{maier}. Apart from inversions of
that kind substitutions one should mention the fact that 3-dimensional
systems, \eg, \eqref{g2g3}, present generically subsystems or
\textit{reductions} of the 4-dimensional ones because differentiations
intertwine equally all the four objects $\vartheta_k$ and $\eta$. In
particular, we display here a version of relations which, besides their
symmetrical form, close differentially these objects.

\begin{proposition}\label{T1}
The canonical Jacobi's $\vartheta$-constant series satisfy the closed
differential identities upon adjoining the Weierstrass $\eta$-series:
\begin{equation}\label{var}
\begin{aligned}
\frac{d\vartheta_2}{d\tau}&=
\frac{\ri}{\pi}\bigg\{\eta+\frac{\pi^2}{12}\,
\big(\vartheta_3^4+\vartheta_4^4
\big)\bigg\}\vartheta_2\,,\\[0.4em]
\frac{d\vartheta_3}{d\tau}&=
\frac{\ri}{\pi}\bigg\{\eta+\frac{\pi^2}{12}\,
\big(\vartheta_2^4-\vartheta_4^4 \big)\bigg\}\vartheta_3\,,
\end{aligned}\qquad
\begin{aligned}
\frac{d\vartheta_4}{d\tau}&=
\frac{\ri}{\pi}\bigg\{\eta-\frac{\pi^2}{12}\,
\big(\vartheta_2^4+\vartheta_3^4
\big)\bigg\}\vartheta_4\,,\\[0.3em]
\frac{d\eta}{d\tau}&=\frac{\ri}{\pi}\bigg\{2\,\eta^2-
\frac{\pi^4}{12^2}
\big(\vartheta_2^8+\vartheta_3^8+\vartheta_4^8 \big) \bigg\}\,.
\end{aligned}
\end{equation}
\end{proposition}

Derivation of these formulae uses  computation of the theta-derivatives
through the derivatives of $g_2^{}$, $g_3^{}$, that is
Eqs.~\eqref{g2g3} and \eqref{g23}, followed by applying the symmetrical
identity \eqref{324}.

To all appearances, these identities, direct consequences of standard
relations as they are,  do not appear  explicitly in so extensive
literature on theta-functions though we found first three of them in
Appendix A to  monograph \cite{kiritsis}.

On the other hand, considering \eqref{var} as a dynamical system, its
integration, as we shall see, encounters serious difficulties.
Moreover, system \eqref{darboux}, under the definition \eqref{xyz}, is
not a consequence of Eqs.~\eqref{var} but holds only upon restricting
to the constant level surface \eqref{324}. To put it differently, the
modes of embedding functions $\vartheta,\eta$ into differential systems
are not unique and Eqs.~\eqref{var} require modifications. All this
will be the subject matter of further consideration.

For the reasons given above it is essential to have a comprehensive
description for differential properties of the canonical
$\vartheta,\eta$-series as such. In particular, it is of interest to
find a Lagrangian and Hamiltonian formulation for these systems. This
would provide fresh insight into properties of the theta-constants.
This circle of questions, as applied to an equivalent of the system
\eqref{g2g3}, is addressed in the work \cite{chud} by D.~Chudnovsky \&
G.~Chudnovsky and, to the best of our knowledge, this is the first
paper\footnote{See also important comments on p.~5709--5711 of
\cite{nutku} concerning the Darboux--Halphen system \eqref{darboux} and
its relation to Euler's equations and the Lotka--Volterra system; in
the same place the detailed comments on Poisson structures for the
3-dimensional system \eqref{darboux} are presented. In work
\cite{chakr} their construction by means of multivalued integrals is
discussed and a generalization on Nambu's brackets is proposed.} where
the problem of Hamiltonian treatment for dynamical systems of modular
type was raised. These authors proposed a 4th order differential system
\cite{chud} and its reduction to equations of the order 3. Although
their Hamilton function is rather ingenious and correct, the proposed
reduction for system \eqref{g2g3} is not preserved; in notation of work
\cite{chud} on p.~111 the reduction is defined by the constraint
$\lambda=1$. In other words, this reduction is satisfied only by a
trivial solution. It should also be noticed that system \eqref{g2g3}
and its varieties have a nice interpretation as a halfway between
integrable systems and the ones with a chaotic behavior. However, we do
not touch on this kind of problems here because questions of dynamics
and transitions from `exactly solvable but not completely integrable'
flows to the ergodic ones are the main subject of the work \cite{chud}.
In the same place the extended bibliography is given.

\subsection{The paper content}
The subsequent material is organized as follows. In the next section
(Sect.~\ref{S2}) we discuss the correlation between identities like
\eqref{var} and a little-known Jacobi dynamical system for Legendre's
complete elliptic integrals. We give explanations as to why the change
from differential $\vartheta,\eta$-identities to those that should be
thought of as \textit{defining ODEs} is not a trivial question. We
write down the simplest version of such (`integrable') ODEs.
Section~\ref{S5} is technical; it is devoted to explicit integration of
system~\eqref{main} and Jacobi's system~\eqref{ABab}. Method of
solution invokes standard Legendre's and modular techniques and we
present results in both of these forms: the `linear' $k$- and
`nonlinear' (modular) $\tau$-representation. In Sect.~\ref{S5.5} we
explain how these techniques can be exploited in order to derive the
transcendental multivalued integrals. In Sect.~\ref{S3} we exhibit
explicitly  these objects for all of the systems mentioned above and
use them when constructing Lagrangians and the action functional.
Complete Hamiltonian formulation to the systems under study is
expounded in Sect.~\ref{S4}. The found Poisson structures turn out to
be non-obvious (none are simplectic) and may form compatible pencils;
we also describe the genesis of a rational degenerate Poisson bracket
from a Nambu 4-bracket and possible transitions between various Poisson
brackets. Section~\ref{Oh} contains a generalization; we complete the
theory for the Halphen--Brioschi quadratic ODEs. The last
section~\ref{S5.1} (Appendix) contains some historical remarks on
Jacobi's system.

\section{ODEs defining $\vartheta$-constants\label{S2}}
\subsection{On symmetrical system \eqref{var}\label{S2.1}}
As we mentioned above all the varieties of dynamical systems under
consideration are algebraically related to each other. In this respect
equations \eqref{var} stand out because this system alone represents
the $\eta,\vartheta$-constants. However, point transformations between
dynamical variables are not unique and resulting ODEs for
$\vartheta,\eta$-variables may contain parameters. In connection with
this ambiguity  it is of interest to consider an elegant dynamical
system which was derived by Jacobi. In Jacobi's record \cite{jacobi} it
is as follows:
\begin{equation}\label{ABab}
\left\{
\begin{aligned}
\frac{\partial A}{\partial h}&=2\,A^2B\,,\\[0.3em]
\frac{\partial B}{\partial h}&=b\,A^3\,,
\end{aligned}\qquad\quad
\begin{aligned}
\frac{\partial a}{\partial h}&=-16\,b\,A^2\,,\\[0.3em]
\ds\frac{\partial b}{\partial h}&=a\,b\,A^2\,,
\end{aligned}
\right.
\end{equation}
where $h=\frac14\,\pi\,\ri\,\tau$ and  the restriction
\begin{equation}\label{ab16}
a^2=16\,(1-2\,b)
\end{equation}
is assumed to be imposed. All the information concerning this system
(including solution) has been detailed in the next section and Appendix
(Sect.~\ref{S5.1}) contains additional comments on original motivation
of Jacobi. Jacobi deduced Eqs.~\eqref{ABab} as a set of differential
identities between classical objects of Legendre's `elliptic theory'
\cite{jacobi,halphen,WW}:
\begin{equation}\label{KK}
K(k)=\int\limits_{0\,}^{\,\,1}\!\!\frac{d\lambda}
{\sqrt{(1-\lambda^2)(1-k^2\,\lambda^2)}}\,,\qquad
K'(k)=\int\limits_k^{\,\,1}\!\!\frac{d\lambda}
{\sqrt{(1-\lambda^2)(\lambda^2-k^2)}}\,,
\end{equation}
\begin{equation}\label{EE}
E(k)=\int\limits_{0\,}^{\,\,1}\!
\sqrt{\frac{1-k^2\,\lambda^2}{1-\lambda^2}}\,d\lambda\,,\qquad
E'(k)=\int\limits_{0\,}^{\,\,1}\!
\sqrt{\frac{1-(1-k^2)\,\lambda^2}{1-\lambda^2}}\,d\lambda\,.
\end{equation}
A simple computation, based on the $\vartheta,\eta$-representations of
objects \eqref{KK}--\eqref{EE} appearing in Jacobi's definition  of
variables $\{A,B,a,b\}$---this is Eqs.~\eqref{A1}---shows that
\begin{equation}\label{point1}
A=\vartheta_3^2\,,\quad B=\frac{4}{\pi^2\,\vartheta_3^2}\bigg\{\eta
+\frac{\pi^2}{12}\,\big(\vartheta_2^4-\vartheta_4^4\big)\bigg\},
\quad a=4-8\,\frac{\vartheta_2^4}{\vartheta_3^4}\,,\quad
b=2\,\frac{\vartheta_2^4}{\vartheta_3^4}\,
\frac{\vartheta_4^4}{\vartheta_3^4}\,.
\end{equation}

Now, if we drop out the fix constraint \eqref{ab16} and consider
\eqref{point1} just as a point change in Eqs.~\eqref{ABab}, we shall
not arrive at symmetrical system \eqref{var}. We may also insert into
the change \eqref{point1} some parameters, say
\begin{equation}\label{tmp}
a=4-\boldsymbol{\alpha}\,\frac{\vartheta_2^4}{\vartheta_3^4}\,,\qquad
b=\boldsymbol{\beta}\,\frac{\vartheta_2^4}{\vartheta_3^4}\,
\frac{\vartheta_4^4}{\vartheta_3^4}\,,
\end{equation}
and yield different forms to resulting ODEs but we never get the system
\eqref{var} in this way. (Converse is of course also true: symmetrical
system \eqref{var} does not entail Jacobi's equations \eqref{ABab}).
Any of such ODEs will be integrable in terms of $\vartheta,\eta$-series
since they were obtained from \eqref{ABab} by coordinate changes of
dynamical (phase) variables. The changes are generally algebraic, \ie,
multi-valued in both directions. In this respect Jacobi's system
\eqref{ABab} is not the best variant because choice of the phase
variables, in this case, would lead to replacing the `simple' $K(k)$
with the `cumbrous' $K\! \mbig[6](\!\sqrt{\frac12-\frac12\frac{a}
{\sqrt{a^2+32b}}}\,\mbig[6])$ (see Proposition~\ref{T8} further below).
In other words, the search for a representative defining
Jacobi--Weierstrass' series by a system of ODE's is not a trivially
solvable problem and we need to choose, in some sense, `natural and
optimal' version for such a system (call it canonical one). It must
reflects the principal property of the series, namely, the property of
being uniformizing for other algebraic versions \eqref{darboux},
\eqref{g2g3},  \eqref{ABab}, or the like.

For this purpose, however,  symmetrical form \eqref{var} is apt to be
not a good candidate because it is not amenable to integration and we
failed to find out its complete integral. That such a strong
distinction between systems is inherent in the nature of the case
\eqref{var} will be apparent from the consideration of their algebraic
integrals as algebraic curves in homogeneous coordinates
\mbox{$\vartheta_2\!:\!\vartheta_3\!:\!\vartheta_4$}.

\begin{proposition}\label{P1}
The identities \eqref{var}, being considered as a dynamical system,
have an algebraic integral $U$ given by the following rational function
of $\vartheta$'s:
\begin{equation}\label{Ivar}
U\!\cdot\vartheta_2^4\,\vartheta_3^4\,\vartheta_4^4=
\big(\vartheta_3^4-\vartheta_2^4-\vartheta_4^4\big)^3\,.
\end{equation}
This integral generalizes Jacobi's identity \eqref{324} if $U\ne0$.
\end{proposition}

Turning now the restriction \eqref{ab16} into algebraic integral (see
Eq.~\eqref{I} in Sect.~\ref{S5}), we observe that equation \eqref{I},
under generalization \eqref{tmp}, has genus 9, whereas integral
\eqref{Ivar} is a curve of genus 19. The best we have succeed in
solution of  system \eqref{var} is its partial resolution in terms of
elliptic functions. In a nutshell, this procedure is as follows.

Let us change notation $U\mapsto U^2$ and rewrite integral \eqref{Ivar}
in  form of the elliptic curve\footnote{An analogous transformation to
the fourth powers of $\vartheta$'s in integral \eqref{I} for
Eqs.~\eqref{ABab} leads to a zero genus curve and no elliptic functions
appear in this case (see Sect.~\ref{S5.4}).}
$$
2\,U^2\,\boldsymbol{x}\,\boldsymbol{y}=(\boldsymbol{y}-
\boldsymbol{x}-2)^3\,,\qquad
\boldsymbol{x}=2\,\frac{\vartheta_2^4}{\vartheta_4^4}\,,\quad
\boldsymbol{y}=2\,\frac{\vartheta_3^4}{\vartheta_4^4}\,.
$$
Hence it follows that the pair $(\boldsymbol{x},\boldsymbol{y})$ is
parametrized by Weierstrass' $(\wp,\wpp)$-functions and this curve can
be transformed into the canonical Weierstrassian form
$$
\wpp(\mathfrak{u})^2=4\,\wp^3(\mathfrak{u})-g_2^{}\,\wp(\mathfrak{u})
-g_3^{}\,.
$$
The computation is rather simple and we obtain
\begin{equation}\label{xy}
\boldsymbol{x}=\frac{1}{U}\,\wpp(\mathfrak{u})-\wp(\mathfrak{u})+
\frac{U^2}{12}-1\,,\qquad
\boldsymbol{y}=\frac{1}{U}\,\wpp(\mathfrak{u})+\wp(\mathfrak{u})-
\frac{U^2}{12}+1\,,
\end{equation}
where constants $g_2^{}$, $g_3^{}$ are expressed through the integral
$U$:
$$
g_2^{}=\frac{U^4}{12}-2\,U^2\,,\qquad
g_3^{}=-\frac{U^6}{216}+\frac{U^4}{6}-U^2\,.
$$
Therefore formulae \eqref{xy} substituted into Eqs.~\eqref{var} must
cause this system to become a \mbox{$\tau$-evo}\-lution of the
uniformizer $\mathfrak{u}=\mathfrak{u}(\tau)$. Indeed, after some
algebra we derive that
$$
\frac{36\,U}{\pi\,\ri}\!\cdot\!\frac{1}{\vartheta_4^4}
\frac{d\,\mathfrak{u}}{d\tau}=12\,\wp(\mathfrak{u})-U^2
$$
and therefore
$$
\int\!\!\frac{d\,\mathfrak{u}}{12\,\wp(\mathfrak{u})-U^2}=
\frac{\pi\,\ri}{36\,U} \int\!\!\vartheta_4^4\,d\tau+\mbox{const}\,.
$$
The left hand side of this equation is easily integrated because
$$
\frac{12\,U}{12\,\wp(\mathfrak{u})-U^2}=
\zeta(\mathfrak{u}-\varkappa)-\zeta(\mathfrak{u}+\varkappa)+
2\,\zeta(\varkappa)\,,
$$
where $12\,\wp(\varkappa)=U^2$, $\wpp(\varkappa)=\pm U$, and
$\zeta(\mathfrak{u})$, $\sigma(\mathfrak{u})$ are the standard
Weierstrassian functions associated with the basis
$\wp(\mathfrak{u}),\wpp(\mathfrak{u})$ \cite{halphen,WW}. We  get
$$
\frac{3}{\pi\ri}\ln\!\left\{\frac{\sigma(\mathfrak{u}-\varkappa)}
{\sigma(\mathfrak{u}+\varkappa)}\,\re^{2\zeta(\varkappa)\mathfrak{u}}
\right\}
=\int\!\!\vartheta_4^4\,d\tau+\mbox{const}
$$
but integral in the right hand side requires  further integration of
the system. This last step  is unknown.

If equations \eqref{var} are indeed non-integrable then situation is a
manifestation of the mere fact that the differential identity for a
function and differential equation are not one and the same. The
function $u={-}\frac{d}{dz}\ln\!\left(z^2+\frac4z \right)$ solves the
equation $u''=2\,u^3+z\,u-2$ whose general integral is, however, not
representable in terms of any known functions or integrals of them;
this is the 2nd Painlev\'e transcendent \cite{conte}. Here is a less
trivial example. The function
$u=\frac{d}{dz}\ln\!\left\{\mathrm{Ai}(z)+a\,\mathrm{Bi}(z) \right\}$
contains, like our $\vartheta,\eta$-solutions, special functions and a
free constant; functions $\mathrm{Ai}(z)$, $\mathrm{Bi}(z)$ satisfy the
Airy equation $\psi''=z\,\psi$. Here, we again arrive at the
\mbox{$\mathrm{P}_{\!2}$-transcendent} $u''=2\,u^3-2\,z\,u+1$.

\subsection{An integrable modification of system  \eqref{var}}
Returning to the question of canonical representative for ODEs defining
$\vartheta,\eta$-series, we choose the following modification of
equations \eqref{var}:
\begin{equation}\label{var2}
\begin{aligned}
\frac{d\vartheta_2}{d\tau}&=
\frac{\ri}{\pi}\bigg\{\eta+\frac{\pi^2}{12}\,
\big(\vartheta_3^4+\vartheta_4^4
\big)\bigg\}\vartheta_2\,,\\[0.4em]
\frac{d\vartheta_3}{d\tau}&=
\frac{\ri}{\pi}\bigg\{\eta+\frac{\pi^2}{12}\,
\big(\vartheta_3^4-2\,\vartheta_4^4 \big)\bigg\}\vartheta_3\,,
\end{aligned}\qquad
\begin{aligned}
\frac{d\vartheta_4}{d\tau}&=
\frac{\ri}{\pi}\bigg\{\eta-\frac{\pi^2}{12}\,
\big(2\,\vartheta_3^4-\vartheta_4^4
\big)\bigg\}\vartheta_4\,,\\[0.4em]
\frac{d\eta}{d\tau}&=\frac{\ri}{\pi}\bigg\{2\,\eta^2-\frac{\pi^4}{72}
\big(\vartheta_3^8-\vartheta_3^4\,\vartheta_4^4+\vartheta_4^8 \big)
\bigg\}.
\end{aligned}
\end{equation}
Comprehensive explanation as to why the `defining
$\vartheta,\eta$-equations' should have such a form has been detailed
in work \cite{br2}. System \eqref{var2} has even an (integrable)
extension which is described in the same place. It is of interest to
observe  that all the previous dynamical systems contain in effect only
\textit{squares} of $\vartheta$-constants. For this reason, in the
sequel it will be convenient to renormalize variables $\vartheta,\eta$
and adopt the following notation:
\begin{equation}\label{def}
x=\sqrt{\frac{\pi\ri}{6}}\,\vartheta_2^2\,,\qquad
y=\sqrt{\frac{\pi\ri}{6}}\,\,\vartheta_3^2\,,\qquad
z=\sqrt{\frac{\pi\ri}{6}}\,\,\vartheta_4^2\,,\qquad
u=\frac{2\,\ri}{\pi}\,\eta\,.
\end{equation}
Then Eqs.~\eqref{var2} acquire the form
\begin{equation}\label{main}
\begin{aligned}
\bdot x&=(u+y^2+z^2)\,x\,,\\[0.3em]
\bdot y&=(u+y^2-2\,z^2)\,y\,,
\end{aligned}
\qquad\qquad
\begin{aligned}
\bdot z&=(u-2\,y^2+z^2)\,z\,,\\[0.3em]
\bdot u&=u^2-y^4+y^2z^2-z^4\,,
\end{aligned}
\end{equation}
which, along with the Jacobi system \eqref{ABab}, will be the main
subject of further analysis. Apart from simplicity and the symmetry
$y\rightleftarrows {\pm} z$, there are some additional properties
justifying the study of canonical system \eqref{main}.

First of all, the function $u$, independently of $(x,y,z)$, satisfies
the famous Chazy equation
$$
\dddot{\smash[b]{u}}=6\,(2\,u\,\ddot u-3\,\dot u^2)\,,
$$
(proof is a direct calculation) which cannot be said of $\eta$-solution
to the symmetrical version \eqref{var}. For the latter, the function
$\frac{2\ri}{\pi}\,\eta$ solves this equation only if the $U$-integral
\eqref{Ivar} is equal to zero  (consequence of Proposition~\ref{P1}).
Similarly, functions $y$ and $z$ also satisfy a third (not fourth)
order ODE. This is the known Jacobi \mbox{$C$-equation} \eqref{C}
\cite[p.~186]{jacobi}:
\begin{equation}\label{CC}
C^4(\ln C^3 C_{\tau\tau})_\tau^2=
16\,C^3 C_{\tau\tau}+36\,,\qquad C=\frac 1y\,\mbox{\ \ or\ \ }
\frac1z\,.
\end{equation}
The above mentioned symmetry involves only functions $y$ and $z$ other
than the function $x$. Therefore  general solution $1/x(\tau)$ does not
satisfy this Jacobi's equation. Equations \eqref{main} entail that
functions $x(\tau)$ and $x(\tau)$ times a constant satisfy a common
ODE. Hence, making the transformation $C\mapsto \mathrm{const}\cdot C$
in \eqref{CC}, one infers  that the solution $1/x(\tau)$ satisfies
equation \eqref{CC} wherein 36 should be replaced by a free constant.
One easily derives
$$
\tilde C^4(\ln \tilde C^3 \tilde C_{\tau\tau})_\tau^2-
16\,\tilde C^3 \tilde C_{\tau\tau}=
\bigg(6\,\frac{y^2-z^2}{x^2}\bigg)^{\!\!2},\qquad \tilde C\DEF\frac 1x
$$
but right hand side of this equation is a constant indeed. Explanation
to this fact will be apparent from Sect.~\ref{IM} wherein we give a
complete integral to the system \eqref{main}. Thus the function $\tilde
C$ satisfies the 4th order ODE
$$
\left(\tilde C^4(\ln \tilde C^3 \tilde C_{\tau\tau})_\tau^2\right)_\tau
=\left(16\,\tilde C^3 \tilde C_{\tau\tau}\right)_\tau
$$
which is checked by a straightforward substitution.

\section{Explicit solutions and technicalities\label{S5}}
At first, let us integrate Jacobi's system. From \eqref{ABab} it
follows that $a\,\partial a=-16\,\partial b$ and this equation yields
an algebraic integral that replaces Jacobi's restriction \eqref{ab16}:
\begin{equation}\label{I}
I^2=a^2+32\,b\qquad\Rightarrow\qquad \bdot I\equiv 0\,.
\end{equation}
Therefore $b$ is expressed via the function $a$ which in turn satisfies
a simple differential consequence of \eqref{ABab}, namely, the 3rd
order equation
$$
\frac{a_{\smash{\mathit{hhh}}}^{}}
{\pow{a}{h}{3}}-\frac32\,
\frac{\pow{a}{\mathit{hh}}{2}}{\pow{a}{h}{4}}=
-\frac12\,\frac{a^2+3\,I^2}{(a^2-I^2)^2}\,.
$$
This is a variety of the standard differential equation for Legendre's
modulus $\lambda\DEF k^2(\tau)$:
\begin{equation}\label{lambda}
\frac{\lambda_{\tau\tau\tau}}{\pow{\lambda}{\tau}{3}}-
\frac32\,\frac{\pow{\lambda}{\tau\tau}{2}}
{\pow{\lambda}{\tau}{4}}=
-\frac12\,\frac{\lambda^2-\lambda+1}{\lambda^2(\lambda-1)^2}\,.
\end{equation}
It immediately follows that there is bound to be a linear fractional
change  between variables $a$ and $\lambda$ transforming
Eq.~\eqref{lambda} into equation for $a$ and vice versa. This simple
computation gives
$$
\lambda=\frac{I-a}{2\,I}\,.
$$
Using the well-known $\vartheta$-constant representation for function
$\lambda(h)$ \cite{WW}, we get
\begin{equation}\label{ak}
a=I-2\,I\,\frac{\vartheta_2^4}{\vartheta_3^4}\!
\Big(\mfrac{\alpha\,h+\beta}{\gamma\,h+\delta}\Big),
\end{equation}
where $\{\alpha,\beta,\gamma,\delta\}$ are free constants with
$\alpha\,\delta\ne\beta\,\gamma$. The further integration for the
variables $\{A$, $B\}$ can be continued in two ways. The first one is
to make use of rules for differential computations \eqref{var2} of the
$\vartheta$-series. Applying them to just found expressions for $a(h)$
and $b(h)$, we get expressions for $A(h)$, $B(h)$. The second way is to
linearize the system because any Schwarz's equation is known to be
related to a certain linear ODE. We shall give solutions both in $h$-
and $k$-representations.

\subsection{Associated linear ODEs\label{S5.2}}
Using \eqref{ak}, we have the obvious transformations between pairs
$(a,b)$ and $(k,I)$:
\begin{equation}\label{ab}
a=I-2\,I\,k^2\,,\qquad 8\,b=I^2\,k^2(1-k^2)\,.
\end{equation}
This allows us to bring \eqref{ABab} into the form
\begin{equation}\label{kI}
\bdot A=2\,A^2B\,,\qquad
\bdot B=\frac18\,I^2\,k^2(1-k^2)A^3\,,\qquad
\bdot{\! k}=\frac{1}{2}\,I\,k\,(1-k^2)\,A^2\,,\qquad
\bdot I=0\,,
\end{equation}
where we let the dot above a symbol denote an $h$-derivative. We regard
this system of equations as an intermediate  equivalent of Jacobi's
system \eqref{ABab} because of its  relation to linear ODEs. Indeed, as
it follows from \eqref{kI}, the quantities $A$ and $B$, as functions of
$k$, satisfy the two linear equations
\begin{equation}\label{tk}
\frac{dA}{dk}=\frac4I\frac{1}{(1-k^2)\,k}\,B\,,\qquad
\frac{d B}{dk}=\frac I4\, k\, A
\end{equation}
and their consequences
\begin{equation}\label{tk2}
k\,(k^2-1)\,A_{\mathit{kk}}+(3\,k^2-1)\,A_k+k\,A=0\,,\qquad
k\,(k^2-1)\,B_{\mathit{kk}}-(k^2-1)\,B_k+k\,B=0\,.
\end{equation}
Since $k$ is Legendre's modulus, it is naturally to expect that these
ODEs are integrable in terms of functions \eqref{KK}--\eqref{EE}.

\begin{proposition}\label{P5}
Canonical Legendre's elliptic integrals \textup{\eqref{KK}--\eqref{EE}}
are differentially closed:
%
%
\begin{equation}\label{linear}
\begin{aligned}
\frac{dK}{dk}&=-\frac{K}{k}-\frac{E}{(k^2-1)k}\,,\\[0.3em]
\frac{dE}{dk}&=-\frac{K}{k}+\frac{E}{k}\,,
\end{aligned}\qquad\qquad
\begin{aligned}
\frac{dK'}{dk}&=\frac{kK'}{1-k^2}+\frac{E'}{(k^2-1)k}\,,\\[0.3em]
\frac{dE'}{dk}&=\frac{kK'}{1-k^2}+\frac{kE'}{k^2-1}\,.
\end{aligned}
\end{equation}
This system, being considered as  a dynamical one, has the general
solution
$$
\begin{aligned}
K&=\boldsymbol{\alpha} K(k)-\boldsymbol{\beta} K'(k)\,,\\[0.3em]
E&=\boldsymbol{\alpha} E(k)+\boldsymbol{\beta}\big[E'(k)-K'(k)\big]\,,
\end{aligned}\qquad
\begin{aligned}
K'&=\boldsymbol{\gamma} K(k)+\boldsymbol{\delta} K'(k)\,,\\[0.3em]
E'&=\boldsymbol{\delta} E'(k)+\boldsymbol{\gamma}\big[K(k)-E(k)\big]\,,
\end{aligned}
$$
where $\{\boldsymbol{\alpha},\boldsymbol{\beta},\boldsymbol{\gamma},
\boldsymbol{\delta}\}$ are free constants.
\end{proposition}

Of course, one should bear in mind that the canonical functions
\eqref{KK}--\eqref{EE} themselves are not independent. Rather they
satisfy the Legendre identity
$$
K(k)E'(k)+K'(k)E(k)-K(k)K'(k)=\frac{\pi}{2}\quad \forall k\,,
$$
which is a particular case of the constant level surfaces for  system
\eqref{linear}:
$$
K\,E'+K'\,E-K\,K'=\frac{\pi}{2}\,(
\boldsymbol{\alpha}\,\boldsymbol{\delta}+\boldsymbol{\beta}\,
\boldsymbol{\gamma})\,.
$$
Curiously, this property and Proposition~\ref{P5} seems to have not
been tabulated in the standard texts. The second order differential
consequences of this system are known. Both $K$ and $K'$ satisfy the
same equation
\begin{alignat*}{2}
k\,(k^2-1)\,\frac{d^2\Psi}{dk^2}+(3\,k^2-1)\,\frac{d\Psi}{dk}+k\,
\Psi&=0\qquad\Rightarrow\qquad\Psi=\big\{K(k),\,K'(k)\big\}\\
\intertext{and common equation solvable by functions $E$ and $E'$ reads
as follows} k\,(k^2-1)\frac{d^2\Psi}{dk^2}+(k^2-1)\frac{d\Psi}{dk}-k\,
\Psi&=0 \qquad\Rightarrow\qquad\Psi=\big\{E(k),\,E'(k)-K'(k)\big\}\,.
\end{alignat*}
The two last linear ODEs are not identical to \eqref{tk2} but search
for solutions to Eqs.~\eqref{tk}--\eqref{tk2} is not a difficult task.
In addition to solution pair \eqref{ab}, we obtain that
\begin{equation}\label{AB}
\begin{aligned}
A&=4\,\boldsymbol{\alpha} K(k)+4\, \boldsymbol{\gamma}K'(k)\,,\\[0.3em]
IB&=\boldsymbol{\alpha}\big[E(k)+(k^2-1) K(k)\big]-\boldsymbol{\gamma}
\big[E'(k)-k^2K'(k) \big]
\end{aligned}
\end{equation}
with some free constants $\boldsymbol{\alpha}$, $\boldsymbol{\gamma}$.
We can now combine the `$k$-formulae' \eqref{AB} and $h$-time dynamics
to obtain the complete integral of system \eqref{ABab}.

\subsection{Solution to the Jacobi system\label{S5.4}} Let us
denote
\begin{equation}\label{T}
\T\DEF\frac{\alpha h+\beta}{\gamma h+\delta}\,.
\end{equation}
Then, by virtue of \eqref{lambda},
\begin{equation}\label{kT}
\frac{\alpha h+\beta}{\gamma h+\delta}=\ri\frac{K'(k)}{K(k)}
\qquad \scalebox{1.5}[1]{$\Leftrightarrow$}\qquad k=
\frac{\vartheta_2^2(\T)}{\vartheta_3^2(\T)}\,.
\end{equation}
Make use of the representation for integrals \eqref{KK}--\eqref{EE}
through Jacobi's $\eta,\vartheta$-constants. The canonical formulae for
$K$ and $K'$ are well known:
\begin{equation*}
K(k)=\frac{\pi}{2}\,\vartheta_3^2(h)\,,\qquad
K'(k)=\frac{\pi}{2\,\ri}\,h\,\vartheta_3^2(h)\,,\qquad
k=\frac{\vartheta_2^2(h)}{\vartheta_3^2(h)}\,.
\end{equation*}
One can also show that the second pair $\{E,\,E'\}$ has the following
modular $h$-representation:
\begin{alignat*}{2}
E(k)={}&\frac{2}{\pi}&&\frac{1}{\vartheta_3^2(h)}
\bigg\{\phantom{h}\,\eta(h)+
\frac{\pi^2}{12}\big[\vartheta_3^4(h)+\vartheta_4^4(h)\big]
\bigg\},\\[0.3em]
E'(k)={}&\frac{2\,\ri}{\pi}&&\frac{1}{\vartheta_3^2(h)}
\bigg\{h\,\eta(h)-
\frac{\pi^2}{12}\big[\vartheta_2^4(h)+\vartheta_3^4(h)\big]\,h
-\frac{\pi}{2}\ri \bigg\}.
\end{alignat*}
Modifying these formulae for the general ratio \eqref{T}, we obtain
\begin{equation}\label{temp}
K(k)=\frac{\pi}{2}\,\vartheta_3^2(\T)\qquad\Rightarrow\qquad
\alpha\,K(k)-\ri\,\gamma\,K'(k)=\frac{K(k)}{\gamma\,h+\delta}=
\frac{\pi}{2}\,\frac{\vartheta_3^2(\T)}{\gamma\,h+\delta}\,.
\end{equation}
Adjust the free integration constants in \eqref{AB} with those of
\eqref{T} and \eqref{kT}. Then we may write
\begin{equation}\label{KA}
A=\pm\sqrt{\frac{4\,\ri}{\pi I}}\,
\big\{\alpha\,K(k)-\ri\,\gamma\,K'(k)\big\}
\end{equation}
and therefore
$$
B=\pm\sqrt{\frac{\ri}{4\,\pi\,I^3}}\,\Big\{\alpha\big[E(k)+(k^2-1)
K(k)\big]-\ri\,\gamma\big[E'(k)-k^2K'(k) \big]\Big\}\,.
$$
Passing to the $\vartheta,\eta$-representation, we arrive at the final
form of the sought-for solution.

\begin{theorem}\label{T7}
General solution to the dynamical system of Jacobi \eqref{ABab} has the
form
$$
a=I-2\,I\,\frac{\vartheta_2^4(\T)}{\vartheta_3^4(\T)}\,,\qquad
b=\frac{I^2}{8}\,\frac{\vartheta_2^4(\T)\,\vartheta_4^4(\T)}
{\vartheta_3^8(\T)}\,,\qquad
A=\pm\sqrt{\frac{\pi\ri}{I}}\,
\frac{\vartheta_3^2(\T)}{\gamma\,h+\delta}\,,
$$
$$
B=\pm\sqrt{\frac{\ri\,I}{\pi^3}}\,
\frac{1}{(\gamma\,h+\delta)\,\vartheta_3^2(\T)}
\bigg\{\frac{\pi^2}{12}\!
\left[\vartheta_2^4(\T)-\vartheta_4^4(\T)\right]+
\eta (\T)+\frac\pi2\,\ri\,\gamma\,(\gamma\,h+\delta)\bigg\}\,,
$$
where $\{I,\alpha,\beta,\gamma,\delta\}$ are free constants subjected
to normalization $\alpha\,\delta-\beta\,\gamma=1$.
\end{theorem}

\subsection{Solution to system \eqref{main}\label{IM}}

One integral for Eqs.~\eqref{main} is easily found because $x$ is
absent in three of these equations. Elimination of $u$ shows that the
function
\begin{equation}\label{H}
\pi\,\boldsymbol{I}^2=\frac{y^2-z^2}{x^2}
\end{equation}
is a constant on solutions of \eqref{main}, that is integral. This
integral is much simpler than those we discussed in Sect.~\ref{S2.1}.
As for solutions to system \eqref{main}, these have the most simple
form as against the other equations we consider. We shall give these
solutions in the next theorem. The last fact we should mention here is
a point transformation from Jacobi's equations to the system
\eqref{main}. The simplest way of getting such a transformation is
realized through the `linearizing' systems \eqref{kI}, \eqref{tk} which
can be thought of as intermediate equivalents for Jacobi's one
\eqref{ABab} or \eqref{main}. Explanation and details have been given
in the previous section. From now on we change Jacobi's $h$-notation
and put
$$
\T\DEF\frac{\alpha\, \tau+\beta}{\gamma\,\tau+\delta}
$$
with normalization $\alpha\,\delta-\beta\,\gamma=1$.

\begin{theorem}\label{T2}
The canonical dynamical system \eqref{main} defining
$\vartheta,\eta$-constants and Jacobi's system \eqref{ABab} are
equivalent. They are related through the following point transformation
\begin{alignat}{4}
A&=\frac{1-\ri}{2\,\bI}\,y\,,&\quad
a&=\frac{12}{\pi\,\ri}\,\frac{y^2-z^2}{x^2}\frac{y^2-2\,z^2}{y^2}\,,
\notag\\[-0.5em]
\label{change}\\[-0.7em]
B&=\frac{1+\ri}{2}\,\frac{\bI}{y}\,(u+y^2-2\,z^2)\,,\quad &
b&=-\frac{18}{\pi^2}\,\frac{z^2}{x^4\,y^4}\,(y^2-z^2)^3\,.\notag
\end{alignat}
The  system \eqref{main} has the following general solution:
$$
x=\varepsilon\frac{\vartheta_2^2(\T)}{\gamma\,\tau+\delta}\,,
\qquad
y=\sqrt{\frac{\pi\ri}{6}}\,
\frac{\vartheta_3^2(\T)}{\gamma\,\tau +\delta}\,,
\qquad
z=\sqrt{\frac{\pi\ri}{6}}\,\frac{\vartheta_4^2(\T)}
{\gamma\,\tau+\delta}\,,
\qquad u=\frac{2\,\ri}{\pi}\,\frac{\eta(\T)}{(\gamma\,\tau+\delta)^2}-
\frac{\gamma}{\gamma\,\tau+\delta}\,,
$$
where $\varepsilon\ne0$ is the fourth free constant. With
$\varepsilon=0$, the solution decouples into the two parametric
elementary one:
$$
x=0\,,\qquad y=\frac{\pm 1}{\gamma\,\tau+\delta}\,,\qquad z=
\frac{1}{\gamma\,\tau+\delta}\,,
\qquad u=-\frac{\gamma^2\,\tau+\gamma\,\delta-1}
{(\gamma\,\tau+\delta)^2}\,.
$$
\end{theorem}

\begin{proof}
The most convenient way to get the point transformation is to exploit
the known general solution of Jacobi's $C$-equation \eqref{CC}
\cite[p.~186]{jacobi}:
$$
C^{\smin1}=\sqrt{\frac{\pi\ri}{6}}\,
\frac{\vartheta_k^2\big(\frac{\alpha\,\tau+\beta}{\gamma\,\tau+\delta}
\big)}
{\gamma\,\tau+\delta}
$$
and to pass to the intermediate set of (`linear') variables $(A,B,k,I)$
followed by use of identities \eqref{ab}, formulae
\eqref{kT}--\eqref{KA}, and elimination of  integration constants
$\{\alpha,\beta,\gamma,\delta\}$ appearing in the general solution
given by Theorem~\ref{T7}.  Omitting the computation details, we derive
the  transformation $\{A,B,k,\bI\}\to \{x,y,z,u\}$:
\begin{equation}\label{change1}
x= \frac{1+\ri}{\sqrt{\pi}}\,k\,A\,,\quad
y=(1+\ri)\, \bI A\,,\quad
z^2=2\,\ri\,\bI^2(1-k^2)\, A^2\,,\quad
u=2\, A\big\{B-\ri\,\bI^2(2\,k^2-1)\,A\big\}\,,
\end{equation}
where $12\,\ri\,\bI^2=I$. This change turns \eqref{main} into the
system \eqref{kI}. Inverting this change, we obtain
\begin{equation}\label{change2}
A=\frac{1-\ri}{2\,\bI}\,y\,,\qquad
B=\frac{1+\ri}{2}\,\frac{\bI}{y}\,(u+y^2-2\,z^2)\,,\quad
\bI^2=\frac{1}{\pi}\frac{y^2-z^2}{x^2}\,,\quad k^2=1-\frac{z^2}{y^2}
\end{equation}
and, subsequently, the substitution \eqref{change}. Making use of
solution given in Theorem~\ref{T7}, we get the solution for variables
$\{x,y,z,u\}$.
\end{proof}

Summarizing, an `integrable' modification of the change \eqref{point1}
is not obvious a priori; this is the change \eqref{change}. In turn,
integral $\pi\,\boldsymbol{I}^2x^2=y^2-z^2$ represents a (corrected)
version of complicated integral \eqref{Ivar} and Jacobi's identity
\eqref{324} turns into a constant level surface in the phase space
$(x,y,z,u)$.

\begin{remark}\label{R1}
From the preceding, incidentally, it follows that equations of the
system \eqref{darboux} become now the differential identities for all
solutions of Eqs.~\eqref{var2} (proof is a calculation). Thus, the
Darboux--Halphen system \eqref{darboux} is a \textit{sub}system for
Eqs.~\eqref{var2} but is a \textit{reduction} for system~\eqref{var}.
See also the last sentence  in Appendix (Sect.~\ref{S5.1}).
\end{remark}

\section{Derivation of integrals\label{S5.5}} In order to integrate
system \eqref{ABab} we made use of its first integral \eqref{I}. Having
a complete solution, we can find the two remaining conserved quantities
and the fourth `integral' corresponds to the time shift  $h\mapsto
h+\varepsilon$. The simplest way to derive the integrals is to use the
linear fractional formula \eqref{kT}. Indeed, the $h$-derivative of
this formula gives the equalities
\begin{equation}\label{KK'}
(\gamma\,h+\delta)^2=\left(\frac{d}{dh}\frac{\alpha\,h+\beta}
{\gamma\,h+\delta}\right)^{\!\!\smin1}=
\left(\ri\,\frac{d}{dh}\frac{K'(k)}{K(k)}\right)^{\!\!\smin1}
=\cdots
\end{equation}
and therefore expression
$$
\cdots=\left(\frac I2\,k\,(1-k^2)\,A^2\,\cdot\ri
\frac{d}{dk}\frac{K'(k)}{K(k)}\right)^{\!\!\smin1}\FED\Phi^2(A,B,k,I)
$$
must be a perfect square. Upon rooting, we  get the $h$-linear function
$\gamma\,h+\delta$ with coefficients depending on dynamical variables
$\{A,B,k,I\}\:\scalebox{1.5}[1]{$\Leftrightarrow$}\: \{A,B,a,b\}$. Its
$h$-derivative
$$
\frac{d\Phi}{dh}=
2\,A^2B\,\frac{\partial\Phi}{\partial A}+
\frac18\,I^2k^2\,(1-k^2)\,A^3\,
\frac{\partial\Phi}{\partial B}+
\frac12\, I\,k\,(1-k^2)\,A^2\,\frac{\partial\Phi}{\partial k}=\cdots
$$
yields an $h$-independent constant $\gamma$, that is integral
$$
\cdots=J_1(A,B,a,b)\,.
$$
Doing the same for
$$
(\alpha\,h+\beta)^2=\left(\frac{d}{dh}\frac{\gamma\,h+\delta}
{\alpha\,h+\beta}\right)^{\!\!\smin1}=
\left(\!-\ri\frac{d}{dh}\frac{K(k)}{K'(k)}\right)^{\!\!\smin1}
=\cdots\,,
$$
we get one more integral $J_2(A,B,a,b)\sim\alpha$. Both of these
integrals are independent of each other since  $\{\alpha,\gamma\}$ are
independent constants. All the calculus with objects $K,k,\ldots$ has
been described in the previous section and computations are somewhat
lengthy but routine. We therefore omit them entirely.

\begin{proposition}\label{T8}
The Jacobi system \eqref{ABab} has the only algebraic $($rational\/$)$
integral $I^2=a^2+32\,b$ and the two functionally independent
transcendental integrals
\begin{equation}\label{J12}
\begin{aligned}
J_1&=4\,K(k)\,\cdot\! B-\big\{E(k)+(k^2-1)\,K(k)\big\}\!\cdot
\! A\, I\,,\\[0.3em]
J_2&=4\,K'(k)\!\cdot\! B+\big\{E'(k)-k^2\,K'(k)\big\}\!\cdot\! A\, I\,,
\end{aligned}
\end{equation}
where $\{I,k\}$, if required, can be expressed via $\{a,b\}$ by the
inversion of formulae \eqref{ab}:
$$
I=\sqrt{a^2+32\,b\,}\,,\qquad
k^2=\frac12-\frac12\,\frac{a}{\sqrt{a^2+32\,b\,}}\,.
$$
Integrals $J_1$, $J_2$ are the multi-valued transcendental functions of
dynamical variables $\{a,b\}$ and the linear ones of $\{A,B\}$.
\end{proposition}

Curiously, the `very simple` monomial dynamical system \eqref{ABab} has
rather complicated transcendently algebraic integrals.  Another way of
derivation of integrals exploits the linear equations
\eqref{tk}--\eqref{tk2} and the well-known Wronskian relation for 2nd
order linear ODEs. For example, the $A$-equation in \eqref{tk2} has
$K(k)$ as its particular solution. Therefore
$$
\left\{K(k)\cdot\frac{dA}{dk}-\frac{dK(k)}{dk}\cdot A\right\}
(k^2-1)\,k=\mathrm{const}\,.
$$
Replacing here $\frac{d}{dk}A$ by $B$ through \eqref{tk} and using
rules \eqref{linear}, we arrive again at the integral $J_1$. The choice
of $K'(k)$ for a particular solution produces the second integral $J_2$
in \eqref{J12}.

\section{Integrals and Lagrangian\label{S3}}
\subsection{Conserved quantities}
Lagrangians, Hamiltonians, and Poisson structures for dynamical systems
are known to be closely related to integrals of the corresponding ODEs.
The Hamiltonian formalism for the systems \eqref{ABab} and
\eqref{main}, which we are about to give is based on construction of
conserved quantities and we first tabulate the complete set of such
objects associated with equations \eqref{main}.

\begin{proposition}\label{P2}
The system \eqref{main} has the only algebraic $($rational\/$)$
integral \eqref{H} and the two transcendental multi-valued ones
\begin{equation}\label{Jnew12}
\boldsymbol{J}{\!_1}=\frac1y(u-2\,y^2+z^2)\,
K\Big(\mfrac zy\Big)+
3\,y\,E\Big(\mfrac zy\Big)\,,
\qquad
\boldsymbol{J}{\!_2}=\frac1y(u+y^2+z^2)\,K'\Big(\mfrac zy\Big)-
3\,y\,E'\Big(\mfrac zy\Big)\,,
\end{equation}
that is\/
$\;\bdot{\!\!\!\boldsymbol{J}}{\!_1}=
\;\bdot{\!\!\!\boldsymbol{J}}{\!_2}
\equiv 0$. The integrals satisfy the identity
$$
\boldsymbol{J}{\!_1}K'\Big(\mfrac zy\Big)-\boldsymbol{J}{\!_2}\,
K\Big(\mfrac zy\Big)=\frac32\,\pi\, y\,.
$$
\end{proposition}

\begin{proof} Straightforward  computation with use of
Propositions~\ref{P5}, \ref{T8} and system \eqref{main} itself.
\end{proof}

It should be emphasized here that integrability of any modular system
is always associated with linear ODEs of Fuchsian class; in particular,
with the hypergeometric equations \cite{ohyama,ablowitz3}. This makes
\textit{inevitable} the appearance of transcendentally multi-valued
functions like $K$, $K'$, $E$, $E'$. Transition between such a `linear'
and `modularly nonlinear' $\tau$-representation was explained in
Sect.~\ref{S5.2}.

Insomuch as the famous system \eqref{darboux} has an extensive
literature and applications, it is not without of interest to translate
just obtained integrals into conserved quantities for this system. The
authors of work \cite{nutku} note properly on p.~5709: `These conserved
quantities have not appeared in the literature for over a century even
though a great deal of works has been done in related areas'. These
integrals (including a generalization of \eqref{darboux}) are discussed
in Ref.~\cite{chakr} and implicitly represented there in a form of
`nonalgebraic, transcendental transformation \ldots\ between'
\cite[p.~1755]{chakr} dynamical variables and integrals through
solutions of a hypergeometric equation and its derivatives.

\begin{proposition}\label{T3}
The following is a complete set of $($two\/$)$ conserved quantities for
the Darboux--Halphen system \eqref{darboux}:
\begin{alignat*}{4}
J_1&=\frac{Z}{\sqrt{X-Z\,}}&\,K\bigg(\!\sqrt{\mfrac{X-Y}{X-Z}}\, \bigg)
&+\sqrt{X-Z\,}&\,E\bigg(\!\sqrt{\mfrac{X-Y}{X-Z}}\,
\bigg)\,,\\[0.3em]
J_2&=\frac{X}{\sqrt{X-Z\,}}&\,K'\bigg(\!\sqrt{\mfrac{X-Y}{X-Z}}\,
\bigg) &-\sqrt{X-Z\,}&\,E'\bigg(\!\sqrt{\mfrac{X-Y}{X-Z}}\, \bigg)\,.
\end{alignat*}
\end{proposition}

\begin{proof}
Taking Remark~\ref{R1} into account and using the definitions
\eqref{xyz} and \eqref{def}, we may write system \eqref{main} as
follows
\begin{equation}\label{XYZ}
X=u+y^2+z^2\,,\qquad Y=u+y^2-2\,z^2\,,\qquad Z=u-2\,y^2+z^2\,.
\end{equation}
Hence
$$
u=\frac13\,(X+Y+Z)\,,\qquad y^2=\frac13\,(X-Z)\,,
\qquad z^2=\frac13\,(X-Y)\,.
$$
Substituting this into \eqref{Jnew12}, we get the statement of the
proposition.
\end{proof}

Not so simple but straightforward computation leads transcendental
integrals for dynamical system \eqref{g2g3}. To do this, one uses the
standard notation $P_\nu^\mu$, $Q_\nu^\mu(z)$ for Legendre's functions
with indices $(\nu,\mu)=\big(\frac12,\frac13 \big)$. Recall that both
of these functions satisfy the equation \cite{WW}
$$
(1-z^2)\,\psi''-2\,z\,\psi'+\Big\{\nu(\nu+1)-
\Mfrac{\mu^2}{1-z^2}\Big\}\,\psi=0\,.
$$
Then one can derive and check the following statement.

\begin{proposition}\label{Ig23}
The following expressions
$$
\begin{aligned}
J_1&=\!\!\sqrt[\uproot{1}-2]{g_2^{}w\,}\,\Big\{
P_\nu^\mu(g_3^{}w)-\left(g_3^{}-\tfrac23\,\eta\,g_2^{}\right)\!w
\mspace{5.5mu}
P_{\scalebox{0.7}[1]{$\sss{\!\!-}$}\nu}^\mu(g_3^{}w)\Big\}\,,
\\[0.3em]
J_2&=\!\!\sqrt[\uproot{1}-2]{g_2^{}w\,}\,\Big\{
Q_\nu^\mu(g_3^{}w)-\left(g_3^{}-\tfrac23\,\eta\,g_2^{}\right)\!w
\,Q_{\scalebox{0.7}[1]{$\sss{\!-}$}\nu}^\mu(g_3^{}w)\Big\}\,,
\end{aligned}\qquad
w\DEF\!\!\sqrt[\uproot{2}\leftroot{-1}-2]{g_3^2-
\Mfrac{1}{27}\,g_2^3\,}
$$
provide the two independent transcendental integrals for the
Weierstrass  system \eqref{g2g3}.
\end{proposition}

It is worth noting here that the Halphen system \eqref{darboux} and
Weierstrass' equations \eqref{g2g3} are equivalent by means of the
following transformation $(X,Y,Z)\rightleftarrows(\eta,g_2^{},g_3^{})$
\cite[p.~331]{halphen}:
\begin{alignat*}{4}
6\,\ri\,\eta&{}={}&\pi\,&(X+Y+Z)\,,\\[0.3em]
-3\,g_2^{}&{}={}&\pi^2&(X^2+Y^2+Z^2-XY-XZ-YZ)\,,\\[0.3em]
-54\,\ri\,g_3^{}&{}={}&\pi^3&(2\,X-Y-Z)(2\,Y-X-Z)(2\,Z-X-Y)\,.
\end{alignat*}
As already noted, all these systems  are rationally representable
through the $\vartheta,\eta$-constants, \ie, the phase variables of
system \eqref{var2}. Hence, in addition to \eqref{xyz}, that is the
point change \eqref{XYZ}, we obtain a `correction' of  definition
\eqref{g23}; it determines equations \eqref{g2g3} as a subsystem of
\eqref{var2}:
\begin{equation}\label{Temp}
g_2^{}=\frac{\pi^4}{12}
\big\{\vartheta_3^8-\vartheta_3^4\vartheta_4^4+\vartheta_4^8
\big\}\,,\qquad
g_3^{}=\frac{\pi^6}{432}
\big\{2\,\vartheta_3^4-\vartheta_4^4 \big\}
\big\{\vartheta_3^4+\vartheta_4^4 \big\}
\big\{2\,\vartheta_4^4-\vartheta_3^4 \big\}
\end{equation}
(no $\vartheta_2$ here). These changes entail an equivalence of
integrals: the $\{P_\nu^\mu,\,
Q_\nu^\mu,\,P_{\scalebox{0.7}[1]{$\sss{\!\!-}$}\nu}^\mu,\,
Q_{\scalebox{0.7}[1]{$\sss{\!-}$}\nu}^\mu\}$-objects can be written in
terms of $\{K,\, E,\,K',\, E'\}$ (and vice versa) by the formulae given
above. Of course, this property manifests itself in the fact that
integrals of such a kind may be rewritten solely in terms of
hypergeometric $_2F_1$-functions; we comment this point more fully in
Sect.~\ref{Oh}.

\subsubsection{On the Ramamani system\label{RR}}
By way of illustration to the theory we can also consider
Ramamani's system mentioned in Introduction. This system is satisfied
by certain function series $\mathcal{P}(\tau)$,
$\tilde{\mathcal{P}}(\tau)$, and $\mathcal{Q}(\tau)$ (see formula (3.3)
in \cite{ablowitz2} and formulae (7)--(9) in \cite{guha}). Translating
these definitions into our notation $(\eta,g_2^{},g_3^{})$, we obtain:
$$
\mathcal{P}(\tau)=\frac{4}{\pi^2}\big\{4\,\eta(2\,\tau)-
\eta(\tau)\big\}\,,
\quad
\tilde{\mathcal{P}}(\tau)=
\frac{12}{\pi^2}\big\{2\,\eta(2\,\tau)-\eta(\tau)\big\}\,,
\quad
\mathcal{Q}(\tau)=\frac{4}{5\,\pi^4}
\big\{16\,g_2^{}(2\,\tau)-g_2^{}(\tau)\big\}\,.
$$
Using the duplication rules \eqref{dubl}, notation \eqref{def}, and
definitions \eqref{Temp}, we may, as pointed out above, turn these
formulae into the rational point transformation. One obtains
\begin{equation}\label{pp}
\pi\,\ri\,\mathcal{P}=2\,(u+y^2+z^2)\,,\qquad
\pi\,\ri\,\tilde{\mathcal{P}}=3\,(y^2+z^2)\,,\qquad
\pi^2\mathcal{Q}=36\,y^2z^2
\end{equation}
and this substitution brings the main system \eqref{main} into the
system
\begin{equation}\label{ramam}
\frac{d\mathcal{P}}{d\tau}=
\frac12\,\pi\,\ri\,(\mathcal{P}^2-\mathcal{Q})\,, \qquad
\frac{d\tilde{\mathcal{P}}}{d\tau}=\pi\,\ri\,
(\mathcal{P}\,\tilde{\mathcal{P}}-\mathcal{Q})\,, \qquad
\frac{d\mathcal{Q}}{d\tau}=
2\,\pi\,\ri\,(\mathcal{P}-\tilde{\mathcal{P}})\,\mathcal{Q}\,.
\end{equation}
This is the Ramamani dynamical system \cite[p.~116]{rama},
\cite[(1.8)]{ablowitz2}, \cite[(10)]{guha}. Its theory, including the
search for conserved quantities $J_1,
J_2(\mathcal{P},\tilde{\mathcal{P}},\mathcal{Q})$ (exercise), results
from Proposition~\ref{T3} and an equivalent of substitution \eqref{pp}
written in terms of Darboux--Halphen variables reads as follows:
$$
\pi\,\ri\,\mathcal{P}=2\,X\,,\qquad
\pi\,\ri\,\tilde{\mathcal{P}}=2\,X-Y-Z\,,\qquad
\pi^2\mathcal{Q}=4\,(X-Z)(X-Y)\,.
$$
See Refs.~\cite{ablowitz2,maier} for further information about system
\eqref{ramam}.

\subsection{Action and Lagrangians}
Let us introduce the collective notation for the phase-space
coordinates: $X\DEF(A$, $B$, $a$,
$b)^{\mbox{\tiny$\sss\boldsymbol{\top}$}}$ for  system \eqref{ABab} or
$X\DEF(x$, $y$, $z$, $u)^{\mbox{\tiny$\sss\boldsymbol{\top}$}}$ for
\eqref{main}. We are looking for the action functional
\begin{equation}\label{action}
S=\int\!\!\mathcal{L}(X,\bdot X)\,d\tau
\end{equation}
in the first-order formalism. Then the most general non-singular
Lagrangian has the form
\begin{equation}\label{rho}
\mathcal{L}(X,\bdot X)=\varrho_k^{}(X)\,\bdot X^k-\mathcal{H}(X)\,,
\end{equation}
where $\mathcal{H}(X)$ is a Hamiltonian and $\varrho_k^{}(X)$ define
the symplectic potential $\varrho=\varrho_k^{}(X)\,dX^k$. As usual,
varying \eqref{action}, we get the Hamiltonian equations
\begin{equation*}\label{Omega}
\bdot X^n=\Omega^{\mathit{nk}}(X)\frac{\partial\, \mathcal{H}}
{\partial X^k}\quad\scalebox{1.5}[1]{$\Leftrightarrow$}\quad
\bO_{\mathit{kn}}(X)\bdot X^n=\frac{\partial\, \mathcal{H}}
{\partial X^k}
\qquad
\scalebox{1.5}[1]{$\Leftrightarrow$}\qquad\eqref{ABab}\,,\;
\eqref{main}\,,
\end{equation*}
where $\Omega=\bO^{\smin1}$ is the Poisson bi-vector dual to the
symplectic 2-form
$$
\bO_{\mathit{kn}}=\frac{\partial\varrho_n^{}(X)}{\partial X^k}-
\frac{\partial\varrho_k^{}(X)}{\partial X^n}\,.
$$

There is a great deal of ambiguity concerning the choice of Lagrangians
for a given system of equations. On the other hand, no general method
is known for constructing Lagrangians starting from equations of motion
(the so-called inverse problem of calculus of variations). Our
situation is, however, somewhat special as the systems under
consideration are integrable (in the sense that was explained in
sects.~\ref{S5.4} and \ref{IM}). The latter fact allows us to write
down the following ansatz for the Lagrangian:
\begin{equation}\label{LX}
\mathcal{L}(X,\bdot X)=(\bdot{\mathcal{N}}-1)\,\mathcal{H}+
I_1\,\bdot I_2\,.
\end{equation}
Besides the Hamiltonian $\mathcal{H}(X)$, it involves two additional
independent integrals of motion $I_j=I_j(X)$ and the quantity
$\mathcal{N}=\mathcal{N}(X)$ obeying condition
$\bdot{\mathcal{N}}\equiv 1$. Clearly, $\mathcal{N}$ is a linear
function of $\tau$ modulo integrals of motion, that is
$\mathcal{N}(X)=\tau+\mathrm{const}(\mathcal{H},I_1,I_2)$. We have
constructed such a function in Sect.~\ref{S5.5}. A computation, based
on \eqref{J12} followed by use of \eqref{KA}, \eqref{kI}, and
\eqref{linear}, shows that
$$
\mathcal{N}=-2\,\frac{K(k)}{A\,J_1}\qquad\Rightarrow\qquad
\frac{d\,\mathcal{N}}{dh}\equiv 1\,,
$$
where $J_1$  has been defined in \eqref{J12}. The quantity
$\mathcal{N}(X)$ thus becomes
\begin{equation*}
\mathcal{N}(X)=
-\frac{K\big(\tfrac zy\big)}{y\,\boldsymbol{J}_{\!1}}=
\frac{-K\big(\tfrac zy\big)}{(u-2\,y^2+z^2)\,K\big(\tfrac zy\big)+
3\,y^2\,E\big(\tfrac zy\big)}\qquad(\bdot{\mathcal{N}}\equiv 1)\,.
\end{equation*}
The Lagrangian density is determined up to a total $\tau$-derivative
and therefore its choice is always accompanied by some heuristic
arguments (simplicity of Lagrangians, brackets, etc.) When deriving the
objects above we made use of a `linearized' equivalent to systems
\eqref{ABab}, \eqref{main}, and equation \eqref{kI}. Therefore we
present result in terms of `mixed' phase variables. The most compact
Lagrange function we have found is given by the following statement.

\begin{theorem}\label{P3}
The systems \eqref{ABab}, \eqref{main}, and \eqref{kI} are the
Euler--Lagrange equations for the following Lagrangian  $\mathcal{L}$:
\begin{align}
\mathcal{L}&=
J_1^2\big(\bdot {\mathcal{N}}-1\big)+J_2\,\bdot I-8\,\frac{d}{d\tau}\!
\left(\!\frac{B}{A}K^2\right)
\notag\\[0.3em]
&=4\,\frac{J_1K}{A^2}\!\cdot\!\bdot A
-
2\left\{k\,IK^2+\frac{J_1^2-16B^2K^2}{k\,(k^2-1)\,IA^2} \right\}
\!\cdot\bdot{\!k}
+\left\{ J_2+\frac{2\,K}{AI}\,(J_1-4\,B\,K)\right\}\!\cdot\! \bdot
{I}-J_1^2\,.\label{L}
\end{align}
Here, we omitted indication of argument in Legendre's integral $K(k)$
and expressions for $J_1,J_2(A,B,k,I)$ are taken from
Proposition~$\ref{T8}$. Transitions between variables are described by
substitutions \eqref{change}--\eqref{change2}.
\end{theorem}

\section{Poisson structures\label{S4}}
\noindent The following statement characterizes a non-trivial property
of the dynamical systems under consideration.

\begin{theorem}\label{T4}
Whatever the Hamilton function $\mathcal{H}(X)$ is chosen $($single- or
multi-valued analytic function\/$)$, none of  systems \eqref{var},
\eqref{ABab}, or \eqref{var2}, \eqref{main} does admit a constant
non-degenerate Poisson bracket $\Omega$.
\end{theorem}

\begin{proof}
Denote by $X$ the phase-space coordinate vector for any of these
systems: $\bdot X^j=V^j(X)$. Assuming the availability of the form
$\bO\,\bdot X=\nabla\mathcal{H}(X)$ with a constant matrix $\bO$, we
apply integrability condition  to  equations
$\nabla\mathcal{H}=\bO\,\bdot X$, considered now as equations for the
Hamiltonian $\mathcal{H}$:
\begin{equation}\label{int}
\nabla_{\!k}\mathcal{H}=\bO_{\mathit{kj}}\,V^j\quad\Rightarrow\quad
\nabla_{\!n}(\bO_{\mathit{kj}}\,V^j)=
\nabla_{\!k}(\bO_{\mathit{nj}}\,V^j)
\quad\Rightarrow\quad
\bO_{\mathit{kj}}\,\nabla_{\!n} V^j=
\bO_{\mathit{nj}}\,\nabla_{\!k} V^j\,.
\end{equation}
It follows that \mbox{$\bO\!\cdot\!\partial_{\mbox{\tiny$X$}}\! V$}
must be a symmetric matrix for all $X$. Straightforward computations
show that this property is compatible with vector fields $V$'s defining
the systems \eqref{var}, \eqref{ABab}, and \eqref{var2}, \eqref{main}
if and only if $\bO\equiv 0$.
\end{proof}

This proof gives in fact a criteria for availability of a canonical
symplectic form (given coordinates) and absence of such a bracket
suggests to look for non-canonical one. Insomuch as \eqref{H} is the
only single-valued function integral, we have to take it (or function
of it) as a Hamiltonian. Furthermore, equations of motion do not depend
on choice of the Lagrangian $\mathcal{L}$ but  bracket $\Omega$ does;
even though the Hamilton function $\mathcal{H}(X)$ and coordinates $X$
have been fixed. We thus have to choose in \eqref{LX}, except for
$\mathcal{H}(X)$, the two independent integrals $I_1, I_2(X)$ in order
that the bracket $\Omega$ be simplest. We  put
\begin{equation}\label{L12}
\mathcal{L}=(\bdot {\mathcal{N}}-1)\,\mathcal{H}+(\lambda
\boldsymbol{J}_{\!1})^{\smin1}\;\,\bdot{\!\!\!\boldsymbol{J}}_{\!2}\,,
\end{equation}
where $\lambda$ is an arbitrary constant. Formulae of the previous
section contain all what we need for computation of the bracket
$\Omega$.

\begin{lemma}\label{P4}
Having  integrals of motion $\mathcal{H}$, $I_k$, and the object
$\mathcal{N}$, the Poisson bi-vector $\bO$ is calculated by the
following computational rule:
$$
\bO=\boldsymbol{M}-\boldsymbol{M}^{\mbox{\tiny$\sss\boldsymbol{\top}$}}
\,,\qquad
\boldsymbol{M}_{\!\mathit{kn}}=\nabla_{\!k}\,\mathcal{H}
\cdot\nabla_{\!n}\,\mathcal{N}
+\,
\nabla_{\!k}\,I_1\cdot\nabla_{\!n}\,I_2\,.
$$
\end{lemma}

Now, we insert here
\begin{equation}\label{I12}
I_1\DEF(\lambda \boldsymbol{J}_{\!1})^{\smin1}\,,
\qquad I_2\DEF\boldsymbol{J}_{\!2}
\end{equation}
and use Proposition~\ref{P5}.

\begin{theorem}\label{T5}
Denote $X\DEF(x,y,z,u)^{\mbox{\tiny$\sss\boldsymbol{\top}$}}$. Then
\begin{enumerate}
\item[(1)]
Dynamical system \eqref{main} is Hamiltonian:
$$
\bdot X=\omega\,\nabla\mathcal{H}\,,\qquad\mathcal{H}=
\frac12\,\frac{y^2-z^2}{x^2}\,,
$$
where the degenerate rational $($single-valued\/$)$ Poisson bracket is
as follows
$$
\setlength{\arraycolsep}{0.35em}
\omega=\frac{x}{2\,\mathcal{H}}\!\left(\!\!\!
\begin{array}{cccc}
0&(u{+}y^2{-}2z^2)\,y&(u{-}2y^2{+}z^2)\,z
&u^2{-}y^4{+}y^2z^2{-}z^4\\[0.4em]
{-}(u{+}y^2{-}2z^2)\,y&0&0&0\\[0.4em]
{-}(u{-}2y^2{+}z^2)\,z&0&0&0\\[0.4em]
{-}u^2{+}y^4{-}y^2z^2{+}z^4&0&0&0\\[0.4em]
\end{array} \!\!\right)\!.
$$
\item[(2)]
Non-degenerate but transcendental multi-valued extension of the
$\omega$ is given by the bracket
$$
\Omega=\omega+\lambda\,\tilde\omega\qquad(\mathrm{det}\,\Omega\ne0)\,,
$$
where
$$
\tilde\omega=\frac{2}{\pi}\,K^2\!\left(\!\!\!
\begin{array}{cccc}
0&\Mfrac xy\,z^2&x\,z &x\,M_1\\[0.4em]
-\Mfrac xy\,z^2&0&\Mfrac zy\,(y^2-z^2)&\Mfrac 1y\,M_2\\[2mm]
-x\,z&\Mfrac zy\,(z^2-y^2)&0&z\,M_3\\[0.4em]
-x\,M_1&-\Mfrac 1y\,M_2&-z\,M_3&0\\[0.4em]
\end{array} \!\!\right)
$$
and
$$
M_1\DEF 3\,y^2(EK^{\smin1}-1)^2-z^2\,,\qquad
M_3\DEF y^2(3\,E^2K^{\smin2}-1)+z^2\,,
$$
$$
M_2\DEF 3\,y^4(EK^{\smin1}-1)^2+y^2z^2(6\,EK^{\smin1}-5)+2\,z^4\,.
$$
\item[(3)]
The matrix $\tilde\omega$ is a bracket as well with
$\det\tilde\omega=0$. The brackets $\omega$, $\tilde\omega$ are
compatible to each other and have the following Casimir's
functions:
$$
\omega\,\nabla \boldsymbol{J}_{\!1}=\omega\,\nabla\boldsymbol{J}_{\!2}
\equiv0\,,\qquad
\tilde\omega\,\nabla \mathcal{H}=\tilde\omega\,\nabla\mathcal{N}
\equiv0\,.
$$
The system may thus be treated as  bi-Hamiltonian in the sense of
Magri \textup{\cite{magri}}.
\end{enumerate}
\end{theorem}

Incidentally it should be observed that  degenerate but well-defined
rational bracket $\omega$ is obtained from non-degenerate but
multi-valued bracket $\Omega$ by a passage to the limit $\lambda\to 0$
in transcendental part of the $\Omega$. This procedure can be
interpreted as a formal separability of canonically conjugated pairs
$(\mathcal{H},\mathcal{N})$ and
$(\boldsymbol{J}_{\!1},\boldsymbol{J}_{\!2})$ in Lagrangian
\eqref{L12}. Their commutation relations (algebra of integrals) are
standard:
$$
\big\{\mathcal{H},\,\boldsymbol{J}_{\!1}\big\}_\Omega=
\big\{\mathcal{H},\,\boldsymbol{J}_{\!2}\big\}_\Omega=0, \qquad
\big\{\boldsymbol{J}_{\!2},\,
(\lambda\,\boldsymbol{J}_{\!1})^{\smin1}\big\}_\Omega=1\,.
$$

\begin{remark}\label{R3}
An explicit analog of Theorem~\ref{T5} for Jacobi's system \eqref{ABab}
is obtained with avail of transformation law for tensor
$\Omega(x,y,z,u)\mapsto\widetilde\Omega(A,B,a,b)$ under the coordinate
change $X\DEF(x,y,z,u)^{\mbox{\tiny$\sss\boldsymbol{\top}$}}\mapsto
(A,B,a,b)^{\mbox{\tiny$\sss\boldsymbol{\top}$}}\FED Y$ defined by
Theorem~\ref{T2}. Coordinate form of the transformations reads
$$
\widetilde\Omega^{jp}(Y)=\frac{\partial Y^j}{\partial X^n}\,
\frac{\partial Y^p}{\partial X^m}\,\Omega^{nm}(X)\qquad\Rightarrow\qquad
\widetilde\Omega=\boldsymbol{T}\,\Omega\,
\boldsymbol{T}^{\mbox{\tiny$\sss\boldsymbol{\top}$}}\,,
\quad\boldsymbol{T}_{kn}\DEF\frac{\partial Y^k}{\partial X^n}
$$
and implies equations
$$
\bdot Y^j=\widetilde\Omega^{jp}(Y)\frac{\partial \mathcal{H}}
{\partial Y^p}\qquad
\scalebox{1.5}[1]{$\Leftrightarrow$}\qquad\eqref{ABab}\,.
$$
We do not display here the explicit formulae  since we were unable to
find the compact form to them.
\end{remark}

It is interesting to note that in addition to the algebraic integral
\eqref{I} and the rational (but degenerate) Poisson bi-vector, systems
\eqref{ABab} and \eqref{main} admit a symmetry given by the linear
vector field
$$
\widehat{G}=2\,x\,\partial_x=
A\,\partial_A-B\,\partial_B-2\,a\,\partial_a-4\,b\,\partial_b\,.
$$
This vector field is of course non-Hamiltonian for otherwise it would
be generated by a new integral. The absence of rational integrals of
motion other than $I$ implies that the result of action of
$\widehat{G}$ on $I$ should be a function of $I$. Indeed, one can check
that
$$
\widehat{G}I=-4\,I\,.
$$
Also, once the Hamiltonian form for a dynamical system has been found,
we can determine its invariant volume form $\boldsymbol{V} =
\sqrt{\mathrm{det}\,\bO}\;d^4X$. Calculating determinant of the matrix
$\Omega$, we obtain that
\begin{equation}\label{det}
\mathrm{det}\,\Omega=\frac{4}{\pi^2}\,\lambda^2\boldsymbol{J}_{\!1}^4
\cdot x^6 y^2z^2
\end{equation}
and,  as in the case of Halphen's system \eqref{darboux}
\cite{nutku,chakr}, the volume form is a polynomial function:
$$
\frac{1}{x^3y\,z}\;dx\,dy\,dz\,du\cong \boldsymbol{V}\cong
\frac{1}{A^2b}\;dA\,dB\,da\,db\,.
$$
Clearly, the invariant volume is not unique as one is free to multiply
it on any positive function of integrals.

Let us also comment on the relationship of the degenerate Poisson
structure $\omega$ to the Nambu structure\footnote{The comment has been
added following a suggestion of the anonymous referee who we wish to
thank for that.}. The general Nambu 4-bracket in a four dimensional
space  reads \cite{Nambu,Takh}
\begin{equation}\label{n4}
\{f_1^{},f_2^{},f_3^{},f_4^{}\}=\Xi^{\smin1}(X)\,\varepsilon^{jk\ell n}\,
\nabla_{\!j}f_1^{}\,
\nabla_{\!k}f_2^{}\,\nabla_{\!\ell} f_3^{}\,\nabla_{\!n}f_4^{} \,,
\end{equation}
where  multiplier $\Xi(X)$ transforms as a density  and $\varepsilon$
is the Levi--Civita symbol with $\varepsilon^{1234}=1$. Setting
$\Xi=\sqrt{\det\mathbf{\Omega}}$ times a function of integrals,  one
can see that the rational Poisson bracket $\omega$ appearing  in
Theorem~\ref{T5} is a reduction of the Nambu 4-bracket with respect to
the pair of transcendental integrals $I_1$, $I_2$:
$$
\{f_1^{},\, f_2^{}\}_\omega =  \{f_1^{},\,f_2^{},\,I_2,\,I_1\}\, .
$$
Now the dynamical system \eqref{main} can be viewed as a generalized
Nambu mechanics with the 4-bracket (\ref{n4}) and the triple of
Hamiltonians $\mathcal{H}$, $I_1$, $I_2$ (two of which are
transcendental):
$$
\bdot X=\{X,\mathcal{H},I_2,I_1\} \qquad \Leftrightarrow \qquad \bdot
X=\{X,\mathcal{H}\}_\omega\,.
$$
More explicitly,
\begin{align}\notag
\omega^{\mathit{jk}}&
=\sqrt{\mathrm{det}\,\Omega\,}\,
\varepsilon^{\mathit{jk\ell n}}\cdot\nabla_{\!\ell}I_2\cdot
\nabla_{\!n}I_1=\cdots\\
\intertext{and expression \eqref{det} leads to a
polynomial character of the Nambu bracket:}
&\cdots=\frac{2}{\pi}\,x^3y\,z\,
\varepsilon^{\mathit{jk\ell n}}\cdot\nabla_{\!\ell}
\boldsymbol{J}_{\!1}\cdot
\nabla_{\!n}\boldsymbol{J}_{\!2}\,,\notag
\end{align}
where the integrals $I_k$ and $\boldsymbol{J}_{\!k}$ have been
determined in Eqs.~\eqref{Jnew12} and \eqref{I12}.

We conclude the section with general remarks concerning other
non-constant brackets. All of them are obtainable from each other by
general transformation of the quantities appearing in Lagrangian
\eqref{L12}:
\begin{equation}\label{free}
\!\!\!N\mapsto N+
F_1(H,\boldsymbol{J}_{\!1},\boldsymbol{J}_{\!2})\,,\:
H\mapsto
F_2(H,\boldsymbol{J}_{\!1},\boldsymbol{J}_{\!2})\,,\:
\boldsymbol{J}_{\!2}\mapsto
F_3(H,\boldsymbol{J}_{\!1},\boldsymbol{J}_{\!2})\,,\:
\boldsymbol{J}_{\!2}\mapsto
F_4(H,\boldsymbol{J}_{\!1},\boldsymbol{J}_{\!2})
\end{equation}
(re-normalization of integration constants). This defines a function
freedom of the three variables $(\alpha,\beta,\gamma)
\simeq(H,\boldsymbol{J}_{\!1},\boldsymbol{J}_{\!2})$. On the other
hand, all the dependencies $\Omega(X)$, including possible change of
the Hamilton function $\mathcal{H}$, are determined by the following
modification of the line \eqref{int}:
\begin{equation}\label{int2}
\nabla_{\!n}(\bO_{\mathit{kj}}\,V^j)=
\nabla_{\!k}(\bO_{\mathit{nj}}\,V^j)
\quad\Rightarrow\quad
(\nabla_{\!n}\,\bO_{\mathit{kj}}- \nabla_{\!k}\,\bO_{\mathit{nj}})\,
V^j=
\bO_{\mathit{nj}}\, W_k^j-\bO_{\mathit{kj}}\, W_n^j\,,
\end{equation}
where the tensor field
$$
W_k^j\DEF\frac{\partial V^j}{\partial X^k}
$$
can be thought of as given. Equations \eqref{int2} are a set of partial
differential equations for $\Omega(X)$'s but, thanks to function
freedom mentioned above, we may pass from  old set of variables, say
$(x,y,z,u)$, to the new one $(N,\alpha,\beta,\gamma)$ and thereby turn
these equations into ordinary differential equations in variable $N$.

\begin{theorem}\label{T6}
Denote $\mathit{\Omega}(N;\alpha,\beta,\gamma)\DEF\Omega(x,y,z,u)$ and
matrix $W=W(N;\alpha,\beta,\gamma)$:
$$
W_{\!\mathit{jk}}\DEF \left.\frac{\partial V^j}{\partial X^k}
\right|_{X=X(N;\alpha,\beta,\gamma)}\,,
$$
where $(\alpha,\beta,\gamma)$ are seen as parameters. Then all the
brackets $\Omega(X)=\mathit{\Omega}(N)$ satisfy the linear matrix
dynamical system
\begin{equation}\label{temp}
\frac{d\,\mathit{\Omega}}{dN}= W\mathit{\Omega}+
\mathit{\Omega}\,W^{\mbox{\tiny$\sss\boldsymbol{\top}$}}
\end{equation}
supplemented with the arbitrary initial condition $($bracket\/$)$
$\mathit{\Omega}(0)=\Lambda(\alpha,\beta,\gamma)$.
\end{theorem}

\begin{proof}
With use of antisymmetry $\bO_{\mathit{kj}}=-\bO_{\mathit{jk}}$ and
Jacobi's identity
$\nabla_{\!n}\,\bO_{\mathit{kj}}+\nabla_{\!k}\,\bO_{\mathit{jn}}+
\nabla_{\!j}\,\bO_{\mathit{nk}}=0$ equations \eqref{int2} may be
rewritten as
$V^j\,\nabla_{\!j}\,\bO_{\mathit{nk}}=\bO_{\mathit{kj}}\,W^j_n-
\bO_{\mathit{nj}}\,W^j_k$. Hence $\bdot
\bO{}=-\bO\,W-W^{\mbox{\tiny$\sss\boldsymbol{\top}$}}\,\bO$ and,
subsequently, \eqref{temp} since $\bdot \alpha=\bdot \beta=\bdot
\gamma=0$ and
$\bO\:\bdot{\!\mathit{\Omega}}=-\bdot\bO\,\mathit{\Omega}$.
\end{proof}

By this means function freedom \eqref{free} with the three functions
of
three variables $\alpha=\mathcal{H}(X)$,
$\beta=\boldsymbol{J}_{\!1}(X)$, $\gamma=\boldsymbol{J}_{\!2}(X)$ is
converted into the coefficients of dynamical system \eqref{temp}
(matrix $W$) and its initial condition $\Omega(0)$. As the latter one
may  take any particular bracket; for example, the bracket $\Omega$
from Theorem~\ref{T5}.

\section{A generalization\label{Oh}} Outlined receipt of derivation of
integrals and the `linear objects' like $\mathcal{N}$ is directly
extended to more general (Halphen, Brioschi (1881)) quadratic
homogenous systems \cite{halphen0}
\begin{equation}\label{xyzdot}
\bdot x = x^2+\Xi\,,\qquad
\bdot y = y^2+\Xi\,,\qquad
\bdot z = z^2+\Xi\,,
\end{equation}
$$
\Xi\DEF \boldsymbol{a}\,(y-x)^2+\boldsymbol{b}\,(z-x)^2+
\boldsymbol{c}\,(z-y)^2\,,
$$
associated with a hypergeometric equation of the general type
\begin{equation}\label{2F1}
s\,(s-1)\,\Psi''+\big\{(a+b+1)\,s-c\big\}\,\Psi+a\,b\,\Psi=0\,,
\end{equation}
where prime stands for the $s$-derivative. Integrability of this system
(and its generalizations) in terms of associated linear equations was
considered and established independently in the 1990s by many authors:
Ablowitz et all \cite{ablowitz3}, Ohyama \cite{ohyama}, Harnad \&
MacKay \cite{harnad}; see also Refs.~\cite{chakr, abl}. Parameters
$(\boldsymbol{a},\boldsymbol{b},\boldsymbol{c})$ are computed via the
hypergeometric ones $(a,b,c)$ (correcting a typo in formula (3.6) of
Ref.~\cite{ohyama}):
$$
4\,\boldsymbol{a}=a\,c+b\,c-2\,a\,b-c\,, \qquad
4\,\boldsymbol{b}=a^2+b^2-a\,c-b\,c+c-1\,,\qquad
4\,\boldsymbol{c}=c^2+2\,a\,b-a\,c-b\,c-c
$$
and base definitions for variables $(x,y,z)$ and relations between them
and the quantities $\tau$, $s$, and $\Psi$ read as follows
$$
\tau=\frac{\tilde\Psi(s)}{\Psi(s)}\,,\qquad
\bdot s=s^c(s-1)^{a+b-c+1}\Psi^2\,,
$$
\begin{alignat*}{4}
x&=\frac12\,\frac{d}{d\tau}\!\ln \Psi^2(s-1)^{a+b-c+1}s^c&&=
\frac12\,\frac{d}{ds}\!\left\{s^c(s-1)^{a+b-c+1}\Psi^2\right\},
\\[0.3em]
y&=\frac12\,\frac{d}{d\tau}\!\ln \Psi^2(s-1)^{a+b-c+1}s^{c-2}&&=
x-s^{c-1}(s-1)^{a+b-c+1}\Psi^2\,,\\[0.3em]
z&=\frac12\,\frac{d}{d\tau}\!\ln \Psi^2(s-1)^{a+b-c-1}s^c&&=
x-s^c(s-1)^{a+b-c}\Psi^2\,.
\end{alignat*}
See works \cite{ohyama,harnad,chakr} for details. From the formulae
above it follows that
$$
s=\frac{z-x}{z-y}\,,\qquad \bdot s=(x-y)\,s\,,\qquad
\Psi={}_2F_1(a,b;c|s)\,.
$$
In the framework of these definitions we obtain system \eqref{xyzdot}
and can deduce its integrals. Indeed, passing to the general solution
$s=s\big(\frac{\alpha\,\tau+\beta}{\gamma\,\tau+\delta}\big)$, we have,
instead of \eqref{KK'},
$$
(\gamma\,\tau+\delta)^{\smin2}=
\frac{d}{d\tau}\frac{\tilde\Psi(s)}{\Psi(s)}
\sim(x-y)\,s\,\frac{\tilde\Psi'\Psi-\tilde\Psi\Psi'}{\Psi^2}\,.
$$
Hence
$$
(\gamma\,\tau+\delta)\sim\!\frac{\sqrt{s^{c-1}(s-1)^{a+b-c+1}}}
{\sqrt{x-y\,}}\cdot
{}_2F_1(a,b;c|s)\,.
$$
Take the $\tau$-derivative of this expression and make use of the fact
that derivative of a hypergeometric series is another hypergeometric
series \cite{WW}:
$$
\frac{d}{ds}\big\{{}_2F_1(a,b;c|s)\big\}=\frac{a\,b}{c}\cdot
{}_2F_1(a+1,b+1;c+1|s)\,.
$$
We thus obtain the first integral $J_1\sim\gamma$ for
Eqs.~\eqref{xyzdot}:
$$
J_1=
CA\cdot{}_2F_1\!\Big(a,b;c\Big|\Mfrac{z-x}{z-y}\Big)+
CB\cdot
{}_2F_1\!\Big(a+1,b+1;c+1\Big|\Mfrac{z-x}{z-y}\Big)\,,
$$
where
$$
A=(a+b-1)\,x-c\,y-(a+b-c+1)\,z\,,\qquad
B=2\,\frac{a\,b}{c}\,\frac{(x-y)(z-x)}{z-y}\,,
$$
$$
C=(y-x)_{\mathstrut}^{\frac12(a+b-c)}
(z-y)_{\mathstrut}^{\!-\frac12(a+b)}
(z-x)^{\frac12(c-1)}_{\mathstrut}\,.
$$
Assume now that the second (linearly independent of $\Psi$) solution to
\eqref{2F1} has no logarithmic  behavior in the vicinity of point
$s=0$; otherwise we can reorder variables $(x,y,z)$ with the help of
the linear transformation $s\mapsto 1-s$ or $s\mapsto s^{\smin1}$. If
the logarithm presents at each of the points $s=\{0,1,\infty\}$, we
fall into Proposition~\ref{T3}. This case corresponds to parameters
$(a,b,c)=\left(\frac12,\frac12,1\right)$ and is equivalent to system
\eqref{darboux} up to a simple linear transformation \cite{ohyama} of
the triples $(x,y,z)\rightleftarrows (X,Y,Z)$. Then we may take the
following form for the second solution to \eqref{2F1} \cite{WW}:
$$
\tilde\Psi=
s^{1-c}(s-1)^{c-a-b}\cdot{}_2F_1(1-a,1-b;2-c|s)\,.
$$
By repeating the arguments above we obtain the second integral for
Eqs.~\eqref{xyzdot}:
$$
J_2=
\,\tilde{\!C\,}\,\tilde{\!A\,}\cdot{}_2F_1
\!\Big(1-a,1-b;2-c\Big|\Mfrac{z-x}{z-y}\Big)
+
\,\tilde{\!C\,}\,\tilde{\!B\,}\cdot{}_2F_1
\!\Big(2-a,2-b;3-c\Big|\Mfrac{z-x}{z-y}\Big)\,,
$$
where
$$
\tilde{\!A\,}=A+2\,(z+y)\,,\qquad
\tilde{\!B}=2\,\frac{(a-1)(b-1)}{c-2}\,\frac{(x-y)(z-x)}{z-y}\,,
\qquad
\tilde{\!C}=\frac{C}{(z-y)^2}\,.
$$
This completes an integration procedure considered in
Refs.~\cite{abl,ablowitz3,chakr,harnad,ohyama}.

\section{Appendix: The Jacobi  system\label{S5.1}}
Since the late
1850's C.~Borchardt, being the Editor-in-Chief of Crelle's Journal,
began to edit and publish the manuscript material kept after Jacobi's
death in 1851. In particular, in 1857 he published calculations
\cite[p.~383--394]{jacobi} where Jacobi constructed power series
developments for his \mbox{$\theta(z|\tau)$-functions}. The power
\mbox{$\theta$-series} are of interest in their own rights but not a
less remarkable fact is that they produce the nice dynamical systems
integrable in terms of $\vartheta$-constants.

Jacobi introduces the four variables (we keep completely to Jacobi's
notation in \cite[p.~386]{jacobi})
\begin{equation}\label{A1}
A=\frac{2K}{\pi}\,,\qquad
B=\frac{2E}{\pi}-k'^2\frac{2K}{\pi}\,,\qquad
a=4\,(1-2\,k^2)\,,\qquad
b=2\,k^2\,k'^2\,,
\end{equation}
and  shows that these satisfy the dynamical system \eqref{ABab}; in
doing so Jacobi imposes the condition \eqref{ab16} which is of course
an equivalent of the relation $k^2+k'^2=1$ or, which is the same, the
$\vartheta$-identity \eqref{324}. Halphen does not mention system
\eqref{ABab} and, to all appearances, it has not received mention in
the later literature on theta-functions. Jacobi does not restrict his
consideration to variables \eqref{A1} and exhibits what is called
presently the canonical transformations, \ie, transformations of
dynamical variables preserving the form of equations. Here are his
versions of the transformations \cite[p.~387]{jacobi}:
\begin{alignat*}{5}
A&=\frac{2kK}{\pi}\,,&\quad B&=\frac1k\cdot\frac{2E}{\pi}\,,&\quad
a&=-\frac{4(1+k'{}^2)}{k^2}\,,&\quad
b&=-\frac{2\,k'^2}{k^4}\,,\\[0.3em]
A&=\frac{2k'K}{\pi}\,,&\qquad
B&=\frac{1}{k'}\!\left(\frac{2E}{\pi}-\frac{2K}{\pi}\right),&
a&=\frac{4(1+k^2)}{k'{}^2}\,,& b&=-\frac{2\,k^2}{k'^4}\,.
\end{alignat*}
Complete set of differential relations between these and auxiliary
variables $\{k,k',K,E\}$  was written down by Jacobi earlier
\cite[p.~176--177]{jacobi}. As in the previous differential systems
\eqref{darboux} and \eqref{g2g3}, dynamical variables $\{A,B,a,b\}$ are
expressed  through the $\eta,\vartheta$-series rationally; see formulae
\eqref{point1}.

We also note that system \eqref{ABab} is notable for its homogenous
monomial structure. Jacobi exploits  intensively this fact when
deriving the power $\theta$-series; the pages 388--391 of his Werke
\cite{jacobi} contain a lot of useful formulae along these lines.
System \eqref{ABab} is not the only dynamical system that was derived
by Jacobi in connection with \mbox{$\theta$-functions}; see also
\cite[p.~173--190]{jacobi}. Jacobi did not pose a question about
integration of \eqref{ABab} as  ODEs, however earlier, in 1847, he
obtained a complete integral for the 3rd order differential equation
\begin{equation}\label{C}
C^4(\ln C^3C_{\tau\tau})_\tau^2=16\,C^3C_{\tau\tau}-\pi^2
\end{equation}
satisfied by each of the $\vartheta$-constants:
$C=\vartheta(\tau)^{\smin2}$ (Jacobi's notation \cite[p.~179]{jacobi}).
On the other hand, this equation must be a certain consequence of
equations \eqref{var} whose solutions are not only the
$\vartheta,\,\eta$-series. Invoking integral \eqref{Ivar}, we conclude
that Jacobi's equation \eqref{C} is indeed the consequence of equations
\eqref{var} with the proviso that $U=0$. It is also clear that this
condition is a necessary one  in order that the Darboux--Halphen system
\eqref{darboux} be a consequence of symmetrical identities \eqref{var}
as well.


\thebibliography{99}


\bibitem{ablowitz} \textsc{Ablowitz,~M.~J.,  Chakravarty,~S. \&
Takhtajan,~L.~A.} \textit{A self-dual Yang--Mills hierarchy and its
reduction to integrable systems in \mbox{$1+1$} and\/ \mbox{$2+1$}
dimensions}. Comm.\ Math.\ Phys. (1993), {\bf158}, 289--314.

\bibitem{abl} \textsc{Ablowitz,~M.~J., Chakravarty,~S.
\& Halburd,~R.} \textit{The Generalized Chazy Equation from the
Self-Duality Equations}. Stud.\ Appl.\ Math. (1999), {\bf103}(1),
75--88.

\bibitem{ablowitz3} \textsc{Ablowitz,~M.~J.,  Chakravarty,~S. \&
Halburd,~R.} \textit{On Painlev\'e and Darboux--Halphen-Type
Equations}. In: \cite{conte}, p.~573--589.

\bibitem{ablowitz2} \textsc{Ablowitz,~M.~J., Chakravarty,~S. \&
Hahn,~H.} \textit{Integrable systems and modular forms of level $2$}.
J.~Phys.\ {\bf A}: Math.\ Gen. (2006), {\bf39}(50),  15341--15353.

\bibitem{hitchin} \textsc{Atiyah,~M. \& Hitchin,~N.}
\textit{The Geometry
and Dynamics of Magnetic Monopoles}. Princeton University Press:
Princeton (1988).

\bibitem{br2}\textsc{Brezhnev,~Yu.~V.} \textit{Non-canonical extension
of $\theta$-functions and modular integrability of
$\vartheta$-constants}. {\tt http:/\!/arXiv.org/abs/1011.1643}.

\bibitem{ablowitz0} \textsc{Chakravarty,~S., Ablowitz,~M.~J \&
Clarkson, P.~A.} \textit{Reductions of self-dual Yang--Mills fields and
classical systems}. Phys.\ Rev.\ Lett. (1990), {\bf65}(9), 1085--1087.

\bibitem{chakr} \textsc{Chakravarty,~S. \& Halburd,~R.} \textit{First
integrals of a generalized Darboux--Halphen system}. J.~Math.\ Phys.
(2003), {\bf 44}(4), 1751--1762.

\bibitem{chud} \textsc{Chudnovsky,~D.~V. \& Chudnovsky,~G.~V.}
\textit{Note on Eisenstein's system of differential equations: an
example of ``exactly solvable but not completely integrable system of
differential equations''}. In: Lecture Notes in Pure and Applied Math.
(1984), {\bf92}, 99--115. D.~V.~Chudnovsky \& G.~V.~Chudnovsky (eds.),
Dekker.

\bibitem{conte} \textsc{Conte,~R.} (Ed.)
\textit{The Painlev\'e property.
One century later}. CRM Series in Mathematical Physics (1999).
Springer--Verlag: New York  (1999). \mbox{\sc Conte, R.} \textit{The
Painlev\'e Approach to Nonlinear Ordinary Differential Equations,}
77--180.

\bibitem{darboux} \textsc{Darboux,~G.} \textit{Sur la th\'eorie des
coordonn\'ees curvilignes et les syst\`emes orthogonaux}. Annales
scientifiques de l’\'Ecole Normale Sup\'erieure. $2^{\mathrm{e}}$
S\'erie, (1878), {\bf VII}, 101--150.


\bibitem{eisenstein} \textsc{Eisenstein,~G.} \textit{Mathematische
Abhandlungen}. Georg Olms: Hildesheim (1967).

\bibitem{guha} \textsc{Guha,~P. \& Mayer,~D.} \textit{Ramanujan
Eisenstein Series, Fa\'a di Bruno Polynomials and Integrable Systems}.
Max Planck Institute: Preprint (2007), {\bf87}.

\bibitem{nutku} \textsc{G\"umral,~H. \& Nutku,~Y.} \textit{Poisson
structure of dynamical systems with three degrees of freedom}.
J.~Math.\ Phys. (1993), {\bf34}(12), 5691--5723.

\bibitem{halphen0} \textsc{Halphen,~G.-H.} \textit{Sur certains
syst\`emes d’\'equations diff\'erentielles}. Compt.\ Rend.\ Acad. Sci.
Paris (1881), {\bf 92}, 1404--1406.

\bibitem{halphen} \textsc{Halphen,~G.-H.}
\textit{Trait\'e des Fonctions
Elliptiques et de Leurs Applications}. \textit{\textbf{I}}.
Gauthier--Villars: Paris (1886).

\bibitem{harnad} \textsc{Harnad,~J. \& McKay,~J.} \textit{Modular
solutions to equations of generalized Halphen type}. Proc.\ Royal Soc.\
London {\bf A} (2000), {\bf 456}(1994), 261--294.

\bibitem{jacobi} \textsc{Jacobi,~C.~G.~J.} \textit{Gesammelte Werke}.
\textit{\textbf{II}}. Verlag von G.~Reimer: Berlin (1882).

\bibitem{kiritsis} \textsc{Kiritsis,~L.} \textit{Introduction to
superstring theory}. Leuven Notes in Mathematical and Theoretical
Physics. Cornell University Press (1998).

\bibitem{magri} \textsc{Magri,~F.} \textit{A simple model of the
integrable Hamiltonian equation}. J.~Math.\ Phys. (1978), {\bf19}(5),
1156--1162.

\bibitem{maier} \textsc{Maier,~R.~S.} \textit{Nonlinear differential
equations satisfied by certain classical modular forms}. Manuscripta
Mathematica (2011), {\bf134}(1/2), 1--42.

\bibitem{Nambu} \textsc{Nambu,~Y.}
\textit{Generalized Hamiltonian Dynamics}. Phys.\ Rev.\ \textbf{D}
(1973), \textbf{7}(8), 2405--2412.

\bibitem{ohyama} \textsc{Ohyama,~Y.} \textit{Systems of nonlinear
differential equations related to second order linear equations}. Osaka
J.~ Math. (1996), {\bf 33}(4), 927--949.

\bibitem{rama} \textsc{Ramamani,~V.} \textit{On some identities
conjectured by Srinivasa Ramanujan in his lithographed notes connected
with partition theory and elliptic modular functions---their
proofs---interconnection with various other topics in the theory of
numbers and some generalizations}. PhD-Thesis, University of Mysore:
Mysore (1970).

\bibitem{ramanujan} \textsc{Ramanujan,~S.} \textit{On certain
arithmetical functions}. Trans.\ Cambridge Phil.\ Soc. (1916),
{\bf22}(9), 159--184; \textit{Collected papers}. Cambridge University
Press (1927).

\bibitem{Takh} \textsc{Takhtajan,~L.}
\textit{On foundation of the generalized Nambu mechanics}.
Comm.\ Math.\ Phys.  (1994), \textbf{160}, 295--315.

\bibitem{WW} \textsc{Whittaker,~E.~T. \& Watson,~G.~N.} \textit{A
Course of Modern Analysis: An Introduction to the General Theory of
Infinite Processes and of Analytic Functions, with an Account of the
Principal Transcendental Functions}. Cambridge University Press:
Cambridge (1996).
\end{document}